\documentclass[11pt]{article}

\usepackage{latexsym}
\usepackage{amssymb}
\usepackage{amsthm}
\usepackage{amscd}
\usepackage{amsmath}
\usepackage{mathrsfs}
\usepackage[all]{xy}
\usepackage{graphicx}
\usepackage{hyperref}

\usepackage[usenames,dvipsnames]{color}

\theoremstyle{definition}

\newtheorem*{theorem*}{Theorem}
\newtheorem*{conjecture*}{Conjecture}

\newtheorem{theorem}{Theorem}[section]

\theoremstyle{definition}

\newtheorem* {example*}{Example}

\newtheorem{lemma}{Lemma}[section]
\theoremstyle{definition}

\theoremstyle{definition}
\newtheorem* {notation}{Notation}

\newtheorem{proposition}{Proposition}[section]
\newtheorem{corollary}{Corollary}[section]
\newtheorem* {remark}{Remark}
\theoremstyle{definition}

\theoremstyle{definition}

\theoremstyle{definition}

\theoremstyle{definition}

\xyoption{dvips}

\usepackage{fullpage}

\numberwithin{equation}{section}

\newcommand{\One}{{1\hspace{-.14cm} 1}}

\def\modu{\ (\mathrm{mod}\ }

\def\({\left(}
\def\){\right)}

\newcommand{\sgn}{\mathrm{sgn}}

                \newcommand{\cX}{\mathcal{X}}
    \newcommand{\sP}{\Lambda}      \newcommand{\CC}{\mathbb{C}}   \newcommand{\QQ}{\mathbb{Q}}    
    
   \newcommand{\cI}{\mathcal{I}}

    \def\ZZ{\mathbb{Z}} \def\Aut{\mathrm{Aut}}   \def\Ind{\mathrm{Ind}} \def\GL{\mathrm{GL}}   \def\Res{\mathrm{Res}}    \def\spanning{\textnormal{-span}}   
\def\Irr{\mathrm{Irr}}  \def\wt{\mathrm{wt}}
  
\def\diag{\mathrm{diag}}

\newcommand{\hs}{\hspace{1mm}}

\newcommand{\leftexp}[2]{{\vphantom{#2}}^{#1}{#2}}
\newcommand{\ba}{\begin{aligned}}
\newcommand{\ea}{\end{aligned}}
\newcommand{\barr}{\begin{array}}
\newcommand{\earr}{\end{array}}
\newcommand{\be}{\begin{equation}}
\newcommand{\ee}{\end{equation}}
\makeatletter
\renewcommand{\@makefnmark}{\mbox{\textsuperscript{}}}
\makeatother

\allowdisplaybreaks[1]

\UseCrayolaColors

\begin{document}
\title{Generalized Involution Models for Wreath Products}
\author{Eric Marberg \\ Department of Mathematics \\ Massachusetts Institute of Technology \\ \tt{emarberg@math.mit.edu}}
\date{}

\maketitle

\begin{abstract}
We prove that if a finite group $H$ has a generalized involution model, as defined by Bump and Ginzburg, then the wreath product $H \wr S_n$ also has a generalized involution model.   This extends the work of  Baddeley concerning involution models for wreath products.  As an application, we construct a Gelfand model for wreath products of the form $A \wr S_n$ with $A$ abelian, and give an alternate proof of a recent result due to Adin, Postnikov, and Roichman describing a particularly elegant Gelfand model for the wreath product $\ZZ_r \wr S_n$.   We conclude by discussing some  notable properties of this representation and its decomposition into irreducible  constituents, proving a conjecture of Adin, Roichman, and Postnikov's.
\end{abstract}

\section{Introduction}

\def\sign{\mathrm{sign}}
\def\Des{\mathrm{Des}}
\def\Inv{\mathrm{Inv}}
\def\Pair{\mathrm{Pair}}
\def\Inn{\mathrm{Inn}}
\def\Ad{\mathrm{Ad}}

\def\cV{\mathcal{V}}

A \emph{Gelfand model} for a group is a representation equivalent to the multiplicity free sum of all the group's irreducible representations.  In the recent papers \cite{APR2007,APR2008}, Adin, Postnikov, and Roichman describe two beautiful Gelfand models for the symmetric group $S_n$ and the wreath product $\ZZ_r \wr S_n$.  These models are remarkable for their simple combinatorial descriptions, which 
go something as follows.  

The Gelfand model for $S_n$ in \cite{APR2007} coincides with the one for $\ZZ_r \wr S_n$ in \cite{APR2008} when $r=1$, so for the moment we discuss only this second model.  We view $\ZZ_r \wr S_n$ as the set of generalized $n\times n$ permutation matrices with nonzero entries given by $r$th roots of unity   
 and define
\[
\cV_{r,n } = \QQ\spanning \left\{ C_\omega : \omega \in \ZZ_r \wr S_n,\ \omega^T = \omega\right\}
\] to be a vector space spanned by 
the symmetric matrices in $\ZZ_r \wr S_n$.  Adin, Postnikov, and Roichman define a representation $\rho_{r,n}$ of $\ZZ_r \wr S_n$ in $\cV_{r,n}$ by the formula
\be\label{intro} 
\rho_{r,n}(g) C_\omega ={\sign_{r,n}(g, \omega)}_{} \cdot C_{g \omega g^T},\qquad\text{for }g,\omega \in \ZZ_r \wr S_n \text{ with }\omega^T = \omega\ee
where 
$\sign_{r,n}(g, \omega)$ is a coefficient taking values in $\{\pm 1\}$.   If $s_1,\dots,s_{n-1} \in \ZZ_r \wr S_n$ correspond to the simple reflections in $S_n$ and $s_0 \in \ZZ_r \wr S_n$ is the diagonal matrix $\diag\( \zeta_r, 1,\dots,1\)$ with $\zeta_r=e^{2\pi i /r}$, then 
\[ \sign_{r,n}(s_i,\omega) = \left\{ \barr{ll} -1, &\text{if }|\omega|(i) = i+1\text{ and }|\omega|(i+1)=i,
\\ 1,&\text{otherwise},\earr\right. \qquad\text{for }1\leq i < n\]
where $|\omega| \in S_n$ denotes the permutation corresponding to the matrix formed by replacing each entry of the matrix $\omega$ with its absolute value, 
 and 
\[ \sign_{r,n}(s_0,\omega) =\left\{\barr{ll} -1,&\text{if $\omega_{11} = \zeta_r^{-1}$ and $r$ is even}, \\ 1,&\text{otherwise}.\earr\right.\]
Theorem 1.2 in \cite{APR2008} asserts that the representation $\rho_{r,n}$ is in fact a Gelfand model for $\ZZ_r \wr S_n$.

Adin, Postnikov, and Roichman provide a largely combinatorial proof of this result.  Their strategy is first to find a formula for the character of any Gelfand model for $\ZZ_r \wr S_n$.  They then prove that the given map is a representation, compute its character, and check that this matches their first formula.  
This approach has the merit of hiding much of the messier representation theory in the background, behind some powerful combinatorial machinery.  
Such a combinatorial method of proof comes at a cost, however.  Besides requiring some detailed and occasionally technical calculations, it does not give us a very clear idea of what motivated the construction of these Gelfand models, or of what accounts for their particular elegance.  As a consequence, one does not know how various subrepresentations of $\rho_{r,n}$ explicitly decompose into irreducible constituents, and it is not evident how we might extend the Gelfand model for $\ZZ_r\wr S_n$, either to wreath products with other groups in place of $\ZZ_r$ or to the complex reflection subgroups $G(r,p,n) \subset \ZZ_r \wr S_n$.  This work arose as attempt to answer the former question of origin, and by extension to address these subsequent problems.

In the special case when $r=1$ and the wreath product $\ZZ_r \wr S_n$ coincides with $S_n$, 
our Gelfand model arises from an \emph{involution model} for $S_n$.  By this, we mean a set of linear characters $\{\lambda_i : C_{S_n}(\omega_i)\to \CC\}$ where $\omega_i$ are representatives of the distinct conjugacy classes of involutions in $S_n$, such that each irreducible character of $S_n$ appears as a constituent with multiplicity one of the sum of induced characters $\sum_i \Ind_{C_{S_n}(\omega_i)}^{S_n}(\lambda_i)$.  In the brief note \cite{IRS91}, Inglis, Richardson, and Saxl describe an involution model for $S_n$ which naturally corresponds to the representation $\rho_{r,n}$ with $r=1$.  Given this observation, a description of how the Gelfand model in \cite{APR2007} decomposes come for free.

Addressing the  case for general $r$ requires more effort on our part.  The work of Baddeley in \cite{B91-2} gives an important clue as to what our answers should look like.  That paper shows how to construct an involution model for the wreath product $H \wr S_n$ when an involution model exists for the finite group $H$. When $H = \ZZ_2$, Baddeley's construction gives rise to the Gelfand model $\rho_{2,n}$.  For all $r>2$, however, the representation $\rho_{r,n}$ does not correspond to an involution model.  In particular, for $r>2$ the symmetric matrices in $\ZZ_r \wr S_n$ are not all involutions and the group $H=\ZZ_r$ does not itself possess an involution model.  

Nevertheless, the Gelfand model $\rho_{r,n}$ does arise from a similar construction.  To describe this precisely, we make use of the definition by  Bump and Ginzburg  in \cite {BG2004} of a  \emph{generalized involution model}.  As one of our main results, we extend Baddeley's work in \cite{B91-2} to prove that if a finite group $H$ has a generalized involution model then so does $H \wr S_n$.  As an application of this result, we construct generalized involution models for $\ZZ_r \wr S_n$ and give a simple, alternate proof that $\rho_{r,n}$ is a Gelfand model. 

The rest of this paper is organized as follows.  Section \ref{prelim} defines a generalized involution model for a group and provides some useful preliminary results.  In Section \ref{S_n} we review the content of \cite{IRS91} and show how it implies the results in \cite{APR2007} concerning Gelfand models for the symmetric group.  In addition, we finish a calculation started in \cite{BG2004} to classify all generalized involution models of the alternating groups.  Section \ref{wreath} contains our main results.  In this section, we extend two theorems in \cite{B91-2} to provide a constructive proof of the following:

\begin{theorem*} If a finite group $H$ has a generalized involution model, then so does  $H\wr S_n$ for all $n\geq 1$.
\end{theorem*}

In Section \ref{application} we apply this general result to give an alternate proof that $\rho_{r,n}$ is Gelfand model for $\ZZ_r \wr S_n$.  We also provide Gelfand models for wreath products of the form $A \wr S_n$, where $A$ is an arbitrary finite abelian group.   Using these constructions, we describe explicitly how the representation $\rho_{r,n}$ decomposes into irreducible constituents.  Specifically, given an involution $\omega \in \ZZ_r \wr S_n$, we say precisely which irreducible representations of $\ZZ_r \wr S_n$ appear as constituents of the subrepresentation generated by the vector $C_\omega \in \cV_{r,n}$. This allows us to prove the following theorem, which implies 
Conjecture 7.1 in \cite{APR2008}.

\begin{theorem*}  Let $\cX$ be a set of symmetric elements in $\ZZ_r \wr S_n$.  If the elements of $\cX$ span a $\rho_{r,n}$-invariant subspace of $\cV_{r,n}$, then the subrepresentation of $\rho_{r,n}$ on this space is equivalent to the multiplicity-free sum of all irreducible $\ZZ_r\wr S_n$-representations whose shapes are obtained from the elements of $\cX$  by the colored RSK correspondence.
\end{theorem*}
%
 
 This information  provides the starting point of the complementary paper \cite{?}, in which we classify the generalized involution models of all finite complex reflection groups.

\section{Preliminaries}\label{prelim}

Below, we introduce the concept of a \emph{generalized involution model} for a finite group, as defined in \cite{BG2004}.  We also state some results due to Kawanaka and Matsuyama \cite{KM1990} and Bump and Ginzburg \cite{BG2004} which relate these models to a generalization of the classical Frobenius-Schur indicator function.

\def\cM{\mathcal{M}}

Throughout, all groups are assumed finite.  Recall that  a \emph{Gelfand model} of a group is a representation equivalent to the multiplicity free sum of all the group's irreducible representations.  
One can always form a Gelfand model by simply taking the direct sum of all irreducible representations,
but 
one usually desires to find some less obvious and more natural means of construction.  
One way of achieving this is through models.  The term ``model'' can mean several different things; 
for our purposes, a \emph{model} for a group $G$ is a set $\{ \lambda_i : H_i \to \CC\}$ of linear characters  of subgroups of $G$ such that
$ \sum_i \Ind_{H_i}^G(\lambda_i)$ is the multiplicity free sum of all irreducible characters of $G$.  By definition, a model gives rise to a Gelfand model which is a monomial representation.

The set of characters forming a model can still appear quite arbitrary, so one often investigates models satisfying some natural additional conditions.  This classic example of this sort of specialization is the \emph{involution model}.  
A model $\{ \lambda_i : H_i \to \CC\}$ is an {involution model} if there exists a set of representatives $\{ \omega_i\}$ of the distinct conjugacy classes of involutions in $G$, such that each subgroup $H_i$ is the centralizer of $\omega_i$ in $G$.  
%
%
%
This definition is made more flexible and hence more useful if we introduce an additional degree of freedom.
 Fix an automorphism $\tau \in \Aut(G)$ such that $\tau^2=1$.  
We denote the action of $\tau$ on $g \in G$ by $\leftexp{\tau}g$ or $\tau(g)$;
the group $G$ then acts on the set of \emph{generalized involutions} 
 \[\cI_{G,\tau} \overset{\mathrm{def}} = \{ \omega \in G : \omega\cdot \leftexp{\tau}{\omega} = 1\}\] by the $\tau$-\emph{twisted conjugation} $g : \omega \mapsto  g\cdot \omega\cdot \leftexp{\tau}{g}^{-1}.$  Let 
 \[ C_{G,\tau}(\omega) = \{ g \in G : g\cdot \omega \cdot \leftexp{\tau}g^{-1}=\omega\}\] denote the stabilizer of $\omega \in \cI_{G,\tau}$ in $G$ under this action.  We call $C_{G,\tau}(\omega)$ the $\tau$-\emph{twisted centralizer} of $\omega$ in $G$ and refer to the orbit of $\omega$ as its \emph{twisted conjugacy class}.

We now arrive at the definition of a \emph{generalized involution model} given by Bump and Ginzburg in \cite{BG2004}.  A generalized involution model for $G$ with respect to $\tau$ is a model $\cM$ for which there exists an injective map $\iota: \cM\to \cI_{G,\tau}$ such that the following hold:
\begin{enumerate}
\item[(a)] Each $\lambda \in \cM$ is a linear character of the $\tau$-twisted centralizer of $\iota(\lambda) \in \cI_{G,\tau}$ in $G$.
\item[(b)] The image of $\iota$ contains exactly one element from each $\tau$-twisted conjugacy class in $\cI_{G,\tau}$.
\end{enumerate}
This is just the definition of an involution model with the word ``centralizer'' replaced by ``twisted-centralizer.'' Indeed, an involution model is simply a generalized involution model with $\tau = 1$.  

\begin{remark}
 The original definition of a generalized involution model in \cite{BG2004} differs from this one in the following way:  in \cite{BG2004}, the set $\cI_{G,\tau}$ is defined as $\{ \omega \in G : \omega \cdot\leftexp{\tau}\omega = z\}$ where $z \in Z(G)$ is a fixed central element  with $z^2=1$.  One can show using Theorems 2 and 3 in \cite{BG2004} that under this definition, any generalized involution model with respect to $\tau,z$ is also a generalized involution model with respect to $\tau',z',$ where $\tau'$ is given by composing $\tau$ with an inner automorphism and $z'=1$.   Thus our definition is equivalent to the one in \cite{BG2004}, in the sense that the same models (that is, sets of linear characters) are classified as generalized involution models.  
 \end{remark}
  
 

We study involution models and generalized involution models, as opposed to other sorts of models, because the groups that can possibly possess them satisfy natural requirements too stringent to encourage indifference to existence questions.  In other words, one often ``expects'' certain reasonable families of groups to have generalized involution models, and this  renders classification questions interesting and tractable.

Let us illustrate this for involution models.  Clearly an involution model exists only if the sum of the degrees of all irreducible characters of $G$ is equal to the number of involutions in $G$.  The Frobenius-Schur involution counting theorem says more: namely, that this condition holds if and only if all the irreducible representations of $G$ are equivalent to real representations.  
Thus, if every irreducible representation of $G$ is realizable, then asking whether $G$ has an involution model is a natural question and one almost expects an affirmative answer.  In truth, the answer is often negative; Baddeley \cite{B91} proved in his Ph.D. thesis that the Weyl groups without involution models are those of type $D_{2n}$ ($n>1$), $E_6$, $E_7$, $E_8$, and $F_4$.  (Vinroot \cite{V} extends this result to show that of the remaining finite irreducible Coxeter groups, only the one of type $H_4$ does not have an involution model.)   However, we see from this classification that at the very least, we have an engaging question on our hands.

Our reason for asking whether a group $G$ has a generalized involution model derives from a generalization of the Frobenius-Schur involution counting theorem due to Bump and Ginzburg \cite{BG2004}.  To state this, 
let $\Irr(G)$ denote the set of irreducible characters of $G$, and for each $\psi \in \Irr(G)$ let $\leftexp{\tau} \psi$ denote the irreducible character $\leftexp{\tau}\psi = \psi \circ\tau$.  
We define the twisted indicator function $\epsilon_\tau : \Irr(G) \to \{-1,0,1\}$ by 
 \[ \epsilon_\tau(\psi) = \left\{
 \ba 
 1, &\quad \text{if $\psi$ is the character of a representation $\rho$ with $\rho(g) = \overline{\rho(\leftexp{\tau}g)}$ for all $g \in G$,} \\
 0, &\quad \text{if $\psi \neq \overline{\leftexp{\tau}\psi}$}, \\ 
 -1,&\quad\text{otherwise}.\ea\right.\]  When $\tau=1$, this gives the familiar Frobenius-Schur indicator function.  Kawanaka and Matsuyama  \cite[Theorem 1.3]{KM1990} prove that  $\epsilon_\tau$ has the formula
 \[ \epsilon_\tau(\psi) = \frac{1}{|G|} \sum_{g \in G} \psi(g \cdot \leftexp{\tau}{g}),\qquad\text{for }\psi \in\Irr(G).\] 
In addition, we have the following result, which appears in a slightly different form as Theorems 2 and 3 in \cite{BG2004}.
 
 \begin{theorem} \label{bg-thm} (Bump, Ginzburg \cite{BG2004}) Let $G$ be a finite group with an automorphism $\tau \in \Aut(G)$ 
 such that $\tau^2=1$.  
 Then the following are equivalent:
 \begin{enumerate}
 \item[(1)] The function $\chi : G\to \QQ$ defined by 
  \[ \chi(g) = | \{ u \in G :  u\cdot \leftexp{\tau}{u} = g\}|,\qquad\text{for }g\in G\] is the multiplicity-free sum of all irreducible characters of $G$. 
   \item[(2)] Every irreducible character $\psi$ of $G$ has $\epsilon_\tau(\psi) = 1$.
 \item[(3)]  The sum $\sum_{\psi \in \Irr(G)} \psi(1)$ is equal to $|\cI_{G,\tau}| = |\{ \omega \in G : \omega \cdot \leftexp{\tau}{\omega} = 1\}|$.
 \end{enumerate}      
\end{theorem}

This theorem motivates Bump and Ginzburg's original definition of a generalized involution model. In explanation, if the conditions (1)-(3) hold, then the dimension of any Gelfand model for $G$ is equal to $\sum_i \(G : C_{G,\tau}(\omega_i)\)$ where $\omega_i$ ranges over a set of representatives of the distinct orbits in $\cI_{G,\tau}$.  The twisted centralizers of a set of orbit representatives in $\cI_{G,\tau}$ thus present an obvious choice for the subgroups $\{H_i\}$ from which to construct a model $\{\lambda_i : H_i \to \CC\}$, and one is naturally tempted to investigate whether $G$ has a generalized involution model with respect to the automorphism $\tau$.    

\def\KK{\mathbb{K}}

Before moving on, we state an observation concerning the relationship between a generalized involution model and a corresponding Gelfand model.  In particular, given $\tau \in \Aut(G)$ with $\tau^2=1$ and a fixed subfield $\KK$ of the complex numbers $\CC$, let
\be\label{cV}\cV_{G,\tau} = \KK\spanning\{ C_\omega : \omega \in \cI_{G,\tau}\}\ee be a vector space over $\KK$ spanned by the generalized involutions of $G$.  We often wish to translate a generalized involution model with respect to $\tau\in \Aut(G)$ into a Gelfand model defined in the space $\cV_{G,\tau}$.  
 The following lemma will be of some use later in this regard.




\begin{lemma}\label{observation} Let $G$ be a finite group with an automorphism $\tau \in \Aut(G)$ such that $\tau^2=1$.  Suppose 
there exists a function $\sign_G : G\times  \cI_{G,\tau} \to \KK$ such that the map
 $\rho : G\rightarrow \GL(\cV_{G,\tau})$ defined by
\be\label{defn-by} \rho(g) C_\omega = \sign_G(g,\omega)\cdot C_{g\cdot \omega\cdot \leftexp{\tau}g^{-1}},\qquad\text{for }g \in G, \ \omega \in \cI_{G,\tau}\ee  is a representation.  Then the following are equivalent:
\begin{enumerate}
\item[(1)] The representation $\rho$ is a Gelfand model for $G$.
\item[(2)] The functions 
\[ \left\{ \barr{rccl} \sign_{S_n}(\cdot, \omega): & C_{G,\tau}(\omega) &\to &\KK  \\ & g & \mapsto & \sign_G(g,\omega) \earr\right\},\] with $\omega$ ranging over any set of orbit representatives of $\cI_{G,\tau}$, form a generalized involution model for $G$.
\end{enumerate}
\end{lemma}

\begin{remark} If $G$ has a generalized involution model $\{ \lambda_i : H_i \to \KK\}$ with respect to $\tau \in \Aut(G)$, then there automatically exists a function $\sign_G : G\times \cI_{G,\tau}\to\KK$ such that $\rho$ is a representation and (1) and (2) hold.  One can construct this function by considering the standard representation attached to the induced character $\sum_i \Ind_{H_i}^G(\lambda_i)$. 
 \end{remark}
 
\begin{proof}
This proof is an elementary exercise involving the definition of a representation and the formula for an induced character, which we leave to the reader.
%
\end{proof}

\begin{notation} In the following sections we employ the following notational conventions:  
\begin{enumerate}
\item[]$ \cI_{G} = \cI_{G,1} = \{ g \in G : g^2 = 1\}$;
\item[] $C_{G}(\omega) = C_{G,1}(\omega) = \{ g \in G : g\omega g^{-1}=\omega\}$; 
\item[] $\One=\One_G$ is the trivial character defined by $\One(g) = 1$ for $g \in G$;

\item[] $\otimes$ denotes the internal tensor product;
\item[] $\odot$ denotes the external tensor product.
\end{enumerate}
Thus, if $\rho,\rho'$ are representations of $G$, then $\rho \otimes \rho'$ is a representation of $G$ while $\rho \odot \rho'$ is a representation of $G\times G$, and similarly for characters.
\end{notation}


\section{Involution Models for Symmetric and Alternating Groups}\label{S_n}

In this section we review what is known of the generalized involution models for the symmetric and alternating groups from \cite{APR2007, BG2004, IRS91}.  Since the symmetric group typically has a trivial center and a trivial outer automorphism group, the group's generalized involution models are always involution models in the classical sense.  In preparation for the next section, we give quickly review the proof of Theorem 1.2 in \cite{APR2007} using the results of \cite{IRS91}. In addition, we extend some calculations in \cite{BG2004} to show that the alternating group $A_n$ has a generalized involution model if and only if $n\leq 7$.  

\subsection{An Involution Model for the Symmetric Group}

Klyachko \cite{K1, K2} and Inglis, Richardson, and Saxl \cite{IRS91} first constructed involution 
models for the symmetric group; additional models for $S_n$ and related Weyl groups appear in \cite{AA2001,A2003,AB2005,B91-2,R1991}.   More recently, Adin, Postnikov, and Roichman \cite{APR2007} describe a simple combinatorial action to define a Gelfand model for the symmetric group.  Their construction turns out to derive directly from the involution model in \cite{IRS91}, and goes as follows.  Let $S_n$ be the symmetric group on $n$ letters and define $\cI_{S_n} =\{ \omega \in S_n  : \omega^2 = 1\}$.  Let 
\[ \cV_{n} =\QQ\spanning\{ C_\omega : \omega \in \cI_{S_n}\} \] be a vector space with a basis indexed by $\cI_{S_n}$.  
For any permutation $\pi \in S_n$, define two sets
\[\ba  \Inv(\pi) &= \{ (i,j) : 1\leq i<j\leq n,\ \pi(i) > \pi(j)\}, \\
\Pair(\pi) &= \{ (i,j) : 1\leq i < j \leq n,\ \pi(i) = j,\ \pi(j) = i\}.\ea\]  The set $\Inv(\pi)$ is the inversion set of $\pi$, and its cardinality is equal to the minimum number of factors needed to write $\pi$ as a product of simple reflections.  In particular, the value of the alternating character at $\pi$ is $\sgn(\pi) =(-1)^{|\Inv(\pi)|}$.  The set $\Pair(\pi)$ corresponds to the set of 2-cycles in $\pi$.

%
Define a map $\rho_n : S_n \rightarrow \GL(\cV_{n})$ by 
\[ \rho_n(\pi) C_\omega =\sign_{S_n}(\pi,\omega)\cdot C_{\pi\omega \pi^{-1}},\qquad\text{for }\pi,\omega\in S_n,\ \omega^2=1,\] where 
\be\label{sign_S_n} \sign_{S_n}(\pi,\omega) = (-1)^{|\Inv(\pi) \cap \Pair(\omega)|}.\ee
Adin, Postnikov, and Roichman \cite{APR2007}  prove the following result.

\begin{theorem} \label{APR} (Adin, Postnikov, Roichman \cite{APR2007}) 
The map $\rho_n$ defines a Gelfand model for $S_n$. 
\end{theorem}

Kodiyalam and Verma first proved this theorem in the unpublished preprint \cite{KV2004}, but their methods are considerably more technical than the ones used in the later work \cite{APR2007}.
We provide a very brief proof of this, using the results of \cite{IRS91}, which follows the strategy outlined in the appendix of \cite{APR2007}.  This will serve as a pattern for later results.

That $\rho_n$ is a representation appears as Theorem 1.1 in \cite{APR2007}.  We provide a slightly simpler, alternate proof of this fact for completeness.

\begin{lemma} 
The map $\rho_n : S_n\rightarrow \GL(\cV_n)$ is a representation.
 \end{lemma}
 
\begin{proof}
It suffices to show that for $\omega \in \cI_{S_n}$ and $\pi_1,\pi_2 \in S_n$, 
\[|\Inv(\pi_1\pi_2)\cap \Pair(\omega)| \equiv |\Inv(\pi_1)  \cap \Pair(\pi_2\omega \pi_2^{-1})| +|\Inv(\pi_2)\cap\Pair(\omega)| \modu 2).\] Let $A^c$ denote the set $\{(i,j) : 1\leq i<j\leq n\} \setminus A$.  The preceding identity then follows by considering the Venn diagram of the sets $\Inv(\pi_1\pi_2)$, $\Pair(\omega)$, and $\Inv(\pi_2)$ and noting that 
\[ |\Inv(\pi_1) \cap \Pair(\pi_2\omega\pi_2^{-1})|  = |\Inv(\pi_1\pi_2) \cap \Pair(\omega) \cap \Inv(\pi_2)^c| + |\Inv(\pi_1\pi_2)^c \cap \Pair(\omega) \cap \Inv(\pi_2)|\] since if $i' = \pi_2(i)$ and $j' = \pi_2(j)$, then we have \[\ba \text{$i<j$ and $(i',j') \in \Inv(\pi_1) \cap \Pair(\pi_2\omega\pi_2^{-1})$}\quad\text{iff}\quad\text{$(i,j) \in \Inv(\pi_1\pi_2) \cap \Pair(\omega) \cap \Inv(\pi_2)^c$,} \\
 \text{$i>j$ and $(i',j') \in \Inv(\pi_1)\cap \Pair(\pi_2\omega\pi_2^{-1})$}\quad\text{iff}\quad\text{$(j,i) \in \Inv(\pi_1\pi_2)^c \cap \Pair(\omega) \cap \Inv(\pi_2)$.}
\ea\]  
\end{proof}

The preceding proof shows that as a map \[ \barr{rccl} \sign_{S_n}(\cdot, \omega): & C_{S_n}(\omega)& \to&  \CC \\
 & \pi & \mapsto &  (-1)^{|\Inv(\pi)\cap \Pair(\omega)|},\earr\] the symbol
  $\sign_{S_n}(\cdot,\omega)$ defines a linear character of the centralizer $C_{S_n}(\omega)$.  To name this character more explicitly, observe that elements of $C_{S_n}(\omega)$ permute the support of $\omega$ and also permute the set of fixed points of $\omega$.  In particular, if $\omega \in \cI_{S_n}$ has $f$ fixed points, then $C_{S_n}(\omega)$ is isomorphic to $(S_2\wr S_k) \times S_f$, where $k = (n-f)/2$ and where the wreath product $S_2\wr S_k$ is embedded in $S_n$ so that the subgroup $(S_2)^k$ is generated by the 2-cycles of $\omega$.  We now have a more intuitive definition of 
  $\sign_{S_n}(\pi,\omega)$.

\begin{corollary}\label{intuitive-def}
The value of $ \sign_{S_n}(\pi,\omega) $ for $\omega \in \cI_{S_n}$ and $\pi \in C_{S_n}(\omega)$ is the signature of $\pi$ as a permutation of the set $\{ i  : 1\leq i \leq n,\  \omega(i) \neq i\}$.
\end{corollary}

\def\cyc{\ }

\begin{proof}
If in cycle notation $\omega = (i_1\cyc j_1)\cdots(i_k\cyc j_k)$ where each $i_t < j_t$, then $C_{S_n}(\omega)$ is generated by permutations of the three forms $\alpha,\beta,\gamma$, where $\alpha = (i_t\cyc i_{t+1})(j_t\cyc j_{t+1})$, $\beta = (i_t\cyc j_t)$, and $\gamma$ fixes $i_1,j_1,\dots, i_k,j_k$.  By inspection, our original definition of $\sign_{S_n}(\pi,\omega) $ agrees with the given formula when $\pi$ is one of these generators, and so our formula holds for all $\pi \in C_{S_n}(\omega)$ since $\sign_{S_n}(\cdot,\omega) : C_{S_n}(\omega) \rightarrow \CC^\times$ is a homomorphism.
\end{proof}

That $\rho_n$ is a Gelfand model now comes as a direct result of  the following lemma, given as Lemma 2 in \cite{IRS91}.   
In this statement, we implicitly identify partitions with their Ferrers diagrams.

\begin{lemma}\label{IRS} (Inglis, Richardson, Saxl \cite{IRS91}) Let $\omega \in S_n$ be an involution fixing exactly $f$ points. 
Then the induced character 
\[ \Ind_{C_{S_n}(\omega)}^{S_n}\(\sign_{S_n}(\cdot,\omega)\)\] 
is the multiplicity free sum of the irreducible character of $S_n$ corresponding to partitions of $n$ with exactly $f$ odd columns.  
\end{lemma}

\begin{corollary}
The linear characters $\left\{\sign_{S_n}(\cdot,\omega) : C_{S_n}(\omega)\to \CC\right\}$,
with $\omega$ ranging over any set of representatives of the conjugacy classes in $\cI_{S_n}$, form an involution model for $S_n$.
\end{corollary}

Theorem \ref{APR} now follows immediately by Lemma \ref{observation}.

\def\Fix{\mathrm{Fix}}

\begin{remark}  
The result in \cite{IRS91} actually concerns the function 
$\sign_{S_n}(\cdot,\omega) \otimes \sgn,$  whose value at $\pi \in C_{S_n}(\omega)$ is the signature of $\pi$ as a permutation of the set $\Fix(\omega) \overset{\mathrm{def}} = \{i : \omega(i) = i\}$. 
Our version follows from the fact that tensoring with the alternating character commutes with induction.
In particular, \cite{IRS91} proves that if $\omega\in \cI_{S_n}$ is an involution with no fixed points, then the induction of the trivial character
\[ \Ind_{C_{S_n}(\omega)}^{S_n}(\One)\] is equal to the multiplicity free sum of the irreducible characters of $S_n$ corresponding to partitions with all even rows.  Proposition \ref{main-prop} gives a generalization of this result.
\end{remark}

\subsection{Generalized Involutions Models for the Alternating Group}

In this section with classify all generalized involutions models for the alternating groups $A_n$.  Bump and Ginzburg consider this example in detail in \cite{BG2004}, but stop just short of a complete classification.  We fill in this gap in their calculations with the following proposition.  Before stating it, observe that  for $n>6$, $A_n$ has a trivial center and a nontrivial outer automorphism, which is unique up to composition with inner automorphisms, given by any conjugation map $g \mapsto xgx^{-1}$ with $x \in S_n - A_n$.

\def\wt{\widetilde}

\begin{proposition}
The alternating group $A_n$ (for $n>2$) has a generalized involution model with respect to an inner automorphism if and only if $n \in \{ 5,6\}$ and with respect to an outer automorphism if and only if $n \in \{3,4,7\}$.
\end{proposition}

\begin{proof}
Propositions 2 and 5 in \cite{BG2004} assert that $A_n$ can have a generalized involution model with respect to the identity automorphism only if $n \in \{5,6,10,14\}$, and with respect to the outer automorphism $g \mapsto (1\cyc2)g (1\cyc2)$ only if $n \in \{ 3,4,7,8,12\}$.  Bump and Ginzburg go on to discuss in \cite{BG2004} how to explicitly construct generalized involution models in the cases $n\in \{3,4,5,6,7\}$. To deal with the remaining cases, let $n \in \{8,10,12,14\}$ and suppose there exists a generalized involution model $\{\lambda_i : H_i\to \CC\}$ with respect to $\tau \in \Aut(A_n)$.  We argue by contradiction.  

By Lemma 5.1 in \cite{?}, we may assume that $\tau =1$ if $\tau$ is inner and that $\tau$ is conjugation by $(1\cyc 2)\in S_n-A_n$ if $\tau$ is not inner.  In the first (respectively, second) case, the subgroups $H_i$ are centralizers in $A_n$ of a set of representatives of the $A_n$-conjugacy classes of involutions in $A_n$ (respectively, $S_n-A_n$). 
Let $\omega \in S_n$ be an involution with two fixed points.  Since $\omega \in A_n$ if $n\equiv 2\modu4)$ and $\omega \in S_n-A_n$ if $n \equiv  0\modu 4)$, it follows that some $H_i$ is conjugate to the subgroup $C_{S_n}(\omega) \cap A_n$.  To prove the proposition, we will show that every character induced to $A_n$ from a linear character of $C_{S_n}(\omega) \cap A_n$ fails to be multiplicity free when $n\in \{8,10,12,14\}$.


To this end, write $n= 2k+2$.  We may assume $\omega = (1\cyc 2)(3\cyc 4)\cdots(2k-1\cyc 2k) \in S_{2k}\subset S_n$; note that $C_{S_n}(\omega) = C_{S_{2k}}(\omega) \times S_2$.  
Now, $C_{S_n}(\omega) \cap A_n$ is a subgroup of $C_{S_n}(\omega)$ of index two, and the larger group's action by conjugation on the degree one characters of the subgroup is trivial.  Therefore each linear character of $C_{S_n}(\omega) \cap A_n$ is obtained by restricting a linear character of $C_{S_n}(\omega)$.  The linear characters of $C_{S_n}(\omega)$ are of the form $\lambda \odot \One$ or $\lambda\odot \sgn$ where 
$\lambda$ is a linear character of $C_{S_{2k}}(\omega)$.  Since $\lambda \odot \sgn$ and $(\lambda\otimes \sgn) \odot\One$ have the same restriction to  $C_{S_n}(\omega) \cap A_n$, we may assume that an arbitrary linear character of $C_{S_n}(\omega) \cap A_n$ is obtained by restricting something of the form $\lambda \odot\One$.  By Mackey's theorem and the transitivity of induction, it follows that  any linear character of $C_{S_n}(\omega) \cap A_n$ induced to $A_n$ is equal to
\be\label{nmf} \ba \Ind_{C_{S_n}(\omega)\cap A_n}^{A_n}\(\Res^{C_{S_n}(\omega)}_{C_{S_n}(\omega)\cap A_n} \( \lambda \odot \One\)\)  &
= \Res_{A_n}^{S_n}\(\Ind_{C_{S_n}(\omega)}^{S_n} \( \lambda \odot\One\)\)\\&
=\Res_{A_n}^{S_n} \(\Ind_{S_{2k}\times S_2}^{S_n}\( \Ind_{C_{S_{2k}}(\omega)}^{S_{2k}} (\lambda) \odot \One\)\)  
\ea\ee for some linear character $\lambda $ of $C_{2k}(\omega)$.  We claim that this is never multiplicity free.

This follows by a calculation.  Note that $C_{S_{2k}}(\omega) \cong S_2 \wr S_k$, where the wreath subgroup $(S_2)^k\subset S_{2k}$ is generated by the 2-cycles of $\omega$.  It follows by Clifford theory that $C_{S_{2k}}(\omega)$ has four distinct linear characters $\lambda_i : C_{S_{2k}}(\omega) \to \CC$ defined by
\[ \ba \lambda_1(\pi) &= 1, \\ 
\lambda_2(\pi) &=\sgn(\pi), \\  
\lambda_3(\pi) &=\text{the signature of $\pi$ as a permutation of the set $\{ \{1,2\}, \{3,4\}, \dots, \{2k-1,2k\}\},$}
\\
\lambda_4(\pi) &= \sgn(\pi) \cdot \lambda_3(\pi), 
\ea\] 
for $\pi \in C_{S_{2k}}(\omega)$.  The following fact proves our claim: for each $i$, the induced character $\Ind_{C_{2k}(\omega)}^{S_{2k}} (\lambda_i)$ has two distinct constituents $\chi, \chi'$ such that $\Ind_{S_{2k}\times S_2}^{S_n}(\chi\odot\One)$ and $\Ind_{S_{2k}\times S_2}^{S_n}(\chi'\odot\One)$ have irreducible constituents indexed either by the same partition or transpose partitions of $n$.  Such characters have the same restriction to $A_n$, and so (\ref{nmf}) is not multiplicity free.

Table 1 below illustrates this situation. 
 The third column from Lemma \ref{IRS} when $i=1,2$ and by a computer calculation using {\tt{GAP}} when $i=3,4$.  We apply Pieri's rule to the third column to obtain a pair of conjugate partitions or a single partition with multiplicity two, which we list in the fourth column.  We recall that Pieri's rule states that if $\chi$ is indexed by a partition $\mu$ of $2k$, then $\Ind_{S_{2k}\times S_2}^{S_n}(\chi\odot\One)$ is the multiplicity free sum of the representations of $S_n$ indexed by all partitions of $n$ obtained by adding two boxes to $\mu$ in distinct columns. Observe that the partitions in the last column are transposes of each other when distinct, which proves our claim above.
\end{proof}
\[ \barr{|llll|} 
\hline \vspace{-3mm} &&&\\
\  n \quad& i\qquad &  \text{Partitions of $2k=n-2$ indexing}   \qquad & \text{Partitions of $n$ given by Pieri's rule\ } \\  & & \text{two constituents of $\Ind_{C_{2k}(\omega)}^{S_{2k}} (\lambda_i)$}\qquad & \text{on induction from $S_{n-2}\times S_2$ to $S_n$} \\ 
 \vspace{-3mm}&&&\\
\hline &&&\\    [-10pt]
\ 8 &  1 & (4,2)\text{ and }(2,2,2) & ( 4,4)\text{ and } (2,2,2,2)  \\
 \vspace{-3mm} &&&\\
  & 2 & (3,3)\text{ and }(2,2,1,1) & (4,3,1)\text{ and } (3,2,2,1)  \\ 
 \vspace{-3mm}  &&& \\
  & 3 & (3,3)\text{ and }(4,1,1) & (4,3,1)\text{ with multiplicity two}  \\ 
   \vspace{-3mm}&&& \\
 & 4 & (2,2,2)\text{ and }(3,1,1,1)  & (4,2,2) \text{ and }  (3,3,1,1)
 \\[-10pt]
&& & \\\hline 
& && \\[-10pt]
\ 10 &  1,2 & (4,4)\text{ and }(2,2,2,2)  &  (5,4,1)\text{ and }  (3,2,2,2,1)  
 \\
  \vspace{-3mm}&&&\\
  & 3 & (4,3,1)\text{ and }(5,1,1,1) &  (5,3,1,1)\text{ with multiplicity two}  \\ 
   \vspace{-3mm} &&&\\
 & 4 & (3,2,2,1)\text{ and }  (4,1,1,1,1) &  (5,2,2,1)\text{ and } (4,3,1,1,1) 
 \\[-10pt]
&& & \\\hline 
& && \\[-10pt]
\  12 &  1 & (6,4)\text{ and } (2,2,2,2,2) &  (6,6)\text{ and } (2,2,2,2,2,2)  \\
 \vspace{-3mm} &&&\\
  & 2 & (5,5)\text{ and }  (3,3,2,2) & (5,5,2)\text{ and } (3,3,2,2,2) \\ 
   \vspace{-3mm} &&&\\
  & 3 & (5,3,1,1)\text{ and } (6,1,1,1,1) & (6,3,1,1,1)\text{ with multiplicity two}  \\ 
   \vspace{-3mm} &&&\\
 & 4 & (4,2,2,1,1)\text{ and } (5,1,\dots,1)  & (6,2,2,1,1)\text{ and } (5,3,1,1,1,1) 
\\[-10pt]
&& & \\\hline 
& && \\[-10pt]
  \ 14 &  1,2 & (6,6) \text{ and } (2,2,2,2,2,2)& (7,6,1)\text{ and } (3,2,2,2,2,2,1) \\
 \vspace{-3mm} &&&\\
  & 3 & (6,3,1,1,1)\text{ and } (7,1,\dots,1) & (7,3,1,1,1,1)\text{ with multiplicity two}\ \  \\ 
   \vspace{-3mm} &&&\\
 & 4 & (5,2,2,1,1,1)\text{ and }(6,1,\dots,1) & (7,2,2,1,1,1)\text{ and } (6,3,1,1,1,1,1)  
 \\ [-10pt] &&&\\ \hline
\earr\]
\[\text{Table 1: Constituents of linear characters induced to $S_n$}\]
%
%

%


\section{Generalized Involution Models for Wreath Products}\label{wreath}

The main goal of this section is to generalize Theorems 1 and 2 and Proposition 3 in \cite{B91-2}.  Together, these extended results show how to construct a generalized involution model for the wreath product $H\wr S_n$ given a  generalized involution model for $H$.  From this construction we will derive a simple proof in the next section of Theorem 1.2 in \cite{APR2008}.

Throughout, we fix a finite group $H$ and a positive integer $n$ and let $G_n =   H \wr S_n$, so that $G_n$ is the semidirect product $G_n =  H^n \rtimes S_n$ where $S_n$ acts on $H^n$ by permuting the coordinates of elements.   We denote the action of $\pi \in S_n$ on $h = (h_1,\dots,h_n) \in H^n$ by 
\[\pi (h) \overset{\mathrm{def}}= \(h_{\pi^{-1}(1)},\dots,h_{\pi^{-1}(n)}\)\]  and write elements of $G_n$ as ordered pairs $(h,\pi)$ with $h \in H^n$ and $\pi \in S_n$.  The group's multiplication is then given by
\[ (h,\pi)(k,\sigma) = (\sigma^{-1}(h)\cdot k, \pi \sigma),\qquad\text{for }h,k \in H^n,\ \pi,\sigma \in S_n.\]  Throughout, we identify  $H^n$ and $S_n$ with the subgroups $\{ (h,1) : h \in H^n\}$ and $\{(1,\pi) : \pi \in S_n\}$ in $G_n$, respectively.



\def\wt{\widetilde}
\def\sP{\mathscr{P}}
\def\sT{\mathscr{T}}

\subsection{Irreducible Characters of Wreath Products}\label{wreath-chars}

To begin, we first review the construction of the irreducible characters of $G_n$.  Our notation mirrors but slightly differs from that in \cite{B91-2}.  Given groups $H_i$ and representations $\varrho _i : H_i \rightarrow \GL(V_i)$, for $i=1,\dots,m$,  let 
\[\barr{c} \bigodot_{i=1}^m \varrho _i: \prod_{i=1}^m H_i \to \GL\(\bigotimes_{i=1}^m V_i\)
\earr\] denote the representation defined for $h_i \in H_i$ and $v_i \in V_i$ by 
\[ \barr{c} \(\bigodot_{i=1}^m\varrho _i\)
(h_1,\dots,h_m) (v_1\otimes \cdots \otimes v_m) = \varrho _1(h_1)v_1 \otimes \cdots \otimes \varrho _m(h_m)v_m
.\earr\] If $\chi_i$ is the character of $\varrho _i$, then we let $\bigodot_{i=1}^m \chi_i$ denote the character of $\bigodot_{i=1}^m \varrho _i$.
%
%

Given a representation $\varrho  : H\rightarrow \GL(V)$, we extend  ${\bigodot_{i=1}^n \varrho }$ to a representation 
of $G_n$ by defining for $h\in H^n$, $\pi \in S_n$, and $v_i \in V$,
\[\barr{c} \(\wt{\bigodot_{i=1}^n \varrho }\)(h,\pi)(v_1\otimes \cdots \otimes v_n)=  \( \varrho\(h_{\pi^{-1}(1)}\)v_{\pi^{-1}(1)}\otimes \cdots \otimes \varrho\(h_{\pi^{-1}(n)}\)v_{\pi^{-1}(n)}\).\earr\]    
\begin{remark}
Check that this formula defines a representation, but note that it differs from the corresponding formula in \cite{B91-2}: there the right hand side is $\(\varrho(h_1) v_{\pi^{-1}(1)}\otimes \cdots \otimes \varrho(h_n) v_{\pi^{-1}(n)}\)$.  This is an artifact of our convention for naming elements of $G_n$, which differs from the one implicitly used in \cite{B91-2}, but which will later make some of our formulas nicer.  
\end{remark}

Let $\sP(n)$ denote the set of integer partitions of $n\geq 0$ and let $\sP = \bigcup_{n=0}^\infty \sP(n)$.  Given $\lambda \in \sP(n)$, let $\rho^\lambda$ denote the corresponding irreducible representation of $S_n$ and write $\chi^\lambda : S_n \to \QQ$ for its character.  We extend the representation $\rho^\lambda$ of $S_n$ to a representation $\widetilde{\rho^\lambda} $ of $G_n$ by setting 
\[ \wt {\rho^\lambda}(h,\pi) = \rho^\lambda(\pi),\qquad\text{for } h \in H^n,\ \pi \in S_n.\]
If $\varrho $ is a representation of $H$ and $\lambda \in \sP(n)$, then we define $\varrho  \wr \lambda$ as the representation of $G_n$ given by 
\[\barr{c} \varrho \wr \lambda \overset{\mathrm{def}} = \(\widetilde {\bigodot_{i=1}^n \varrho } \) \otimes \widetilde {\rho^\lambda}.\earr\]
If $\psi$ is the character of $\varrho $, then we define $\psi \wr \lambda$ as the character of $\varrho  \wr \lambda$.   
We now have the following preliminary lemma.

\begin{lemma}\label{char-vals}
Let $\psi$ be a character of $H$ and let $\lambda \in \sP(n)$.  If the cycles of $\pi \in S_n$ are $(i_1^t\cyc i_2^t\cyc \cdots\ i_{\ell(t)}^t)$ for $t=1,\dots,r$, then
\[(\psi \wr \lambda)( h,\pi) = \chi^\lambda(\pi)  \prod_{t=1}^r \psi\( h_{i_{\ell(t)}^t} \cdots h_{i_2^t}h_{i_1^t} \),
\qquad\text{for }h = (h_1,\dots,h_n) \in H^n.\]
\end{lemma}

\begin{proof}
Suppose $\psi$ is the character of a representation $\varrho$ in a vector space $V$ with a basis $\{v_j\}$.  
Observe that if $h_{i_1},h_{i_2},\dots,h_{i_\ell} \in H$, then
\[ \psi\(h_{i_\ell}\cdots h_{i_2} h_{i_1}\) = \sum_{j_1,j_2,\dots,j_\ell}  \( \varrho\(h_{i_1}\) v_{j_1}\big\vert_{v_{j_{2}}}\)\( \varrho\(h_{i_2}\) v_{j_2}\big\vert_{v_{j_{3}}}\)\cdots \( \varrho\(h_{i_\ell}\) v_{j_\ell}\big\vert_{v_{j_{1}}}\).\]
Therefore, it follows by definition that
\[\ba (\psi \wr \lambda)( h,\pi) &= \chi^\lambda(\pi) \sum_{j_1,\dots,j_n}
\(\varrho\(h_{\pi^{-1}(1)}\)v_{j_{\pi^{-1}(1)}}\)\otimes \cdots \otimes \(\varrho\(h_{\pi^{-1}(n)}\)v_{j_{\pi^{-1}(n)}}\)\big\vert_{(v_{j_1}\otimes \cdots \otimes v_{j_n})}
 \\
 &
= \chi^\lambda(\pi) \sum_{j_1,\dots,j_n}\prod_{i=1}^n \( \varrho\(h_{i}\) v_{j_i}\big\vert_{v_{j_{\pi(i)}}}\)
= \chi^\lambda(\pi)  \prod_{t=1}^r \psi\( h_{i_{\ell(t)}^t} \cdots h_{i_2^t}h_{i_1^t} \).
\ea\]
\end{proof}

Recall that  $\Irr(G)$ denotes the set of irreducible characters of a finite group $G$.
Let $\sP_H$ denote the set of all maps $\theta: \Irr(H) \rightarrow \sP$ and define
\[ \barr{c} \sP_H(n) = \left\{ \theta \in \sP_H: \sum_{\psi \in \Irr(H)} |\theta(\psi)| = n\right\}. \earr\] 
The following classification, which appears in \cite{B91-2} and as Theorem 4.1 in \cite{S1989}, derives from Clifford theory.  Stembrige \cite{S1989} attributes its original proof to Specht \cite{Spe}.

\begin{theorem}\label{wreath-reps} (Specht \cite{Spe}) The set of irreducible characters of $G_n$ is in bijection with $\sP_H(n)$.  In particular, each element of $\Irr(G_n)$ is equal to $\chi_\theta$ for a unique $\theta \in \sP_H(n)$, where \[  \chi_\theta \overset{\mathrm{def}} = \Ind_{S_\theta}^{G_n}\(\bigodot_{\psi \in \Irr(H)} \psi \wr \theta(\psi)\) \qquad\text{and}\qquad S_\theta \overset{\mathrm{def}} = \prod_{\psi \in \Irr(H)} G_{|\theta(\psi)|}.
\] In addition, the degree of the character $\chi_\theta$ is 
\[ \deg(\chi_\theta) = n! \prod_{\psi \in \Irr(H)} \frac{\deg(\psi)^{|\theta(\psi)|} \deg\(\chi^{\theta(\psi)}\)}{|\theta(\psi)|!}.\]
\end{theorem}

All products here proceed in the order of a some fixed enumeration of $\Irr(H)$.  The character $\chi_\theta$ is independent of this enumeration because reordering the factors in $S_\theta$ yields a conjugate subgroup.

\newcommand\sU{\mathscr{U}}

\subsection{Inducing the Trivial Character}

Fix an automorphism $\tau \in \Aut(H)$ with $\tau^2 = 1$.  
In this section, we describe the irreducible constituents of the induced character 
\[ \Ind_{V_k^\tau}^{G_{2k}}(\One),\] where $\One \in \Irr(G_{2k})$ denotes the trivial character of $G_{2k}$, and $V_k^\tau$ denotes a subgroup which will be one of the twisted centralizers in our generalized involution model.

Fix a nonnegative integer $k$, and define $W_k \subset S_{2k}$ as the subgroup
\be\label{W_k} W_k =  \xi(S_2 \wr S_k),\ee where $\xi : S_2\wr S_k \rightarrow S_{2k}$ embeds $S_2\wr S_k$ in $S_{2k}$ such that the wreath subgroup $(S^2)^k\subset S_2 \wr S_k$ is mapped to the subgroup of $S_{2k}$ generated by the simple transpositions $(2i-1\cyc 2i)$ for $i=1,\dots,k$.  In other words, let $W_k$ be the centralizer in $S_{2k}$ of the involution 
\be \label{omega_k}\omega_k \overset{\mathrm{def}}=(1\cyc 2)(3\cyc 4)\cdots (2k-1\cyc 2k) \in S_{2k},\ee where by convention $\omega_0 = 1$.  Next, define $\delta^{\tau}_k(H)$ as the following subgroup of $H^{2k}$:
\[ \delta^{\tau}_k(H) =\left \{ (h_1,\leftexp{\tau}h_1, h_2, \leftexp{\tau} h_2,\dots,h_k,\leftexp{\tau}h_k) : h_i \in H \right\}.\]  Observe that the action of $W_k$ preserves $\delta^\tau_k(H)$, and let $V_k^\tau$ denote the subgroup of $G_{2k}$ given by 
\[ V_k^\tau =  \delta^{\tau}_k(H)\cdot  W_k= \left\{ (h,\pi) \in G_{2k} : ,h \in \delta^\tau_k(H),\ \pi \in W_k\right\}.\]  
This subgroup will be one of the key building blocks used to construct the twisted centralizers whose linear characters will comprise a generalized involution model for $G_n$.  In fact, the critical step in constructing a model for $G_n$ from a model for $H$ will be to determine the irreducible constituents of the character of $G_n$ induced from the trivial character of the subgroup $V_k^\tau$.  The following two lemmas take care of some the calculations needed to compute this.

\begin{lemma}\label{lem1}
Let $\psi$ be an irreducible character of $H$ with $\epsilon_\tau(\psi) = \pm 1$ and let $\lambda \in \sP({2k})$.  Then
\[ \left\langle \One, \Res_{V_k^\tau}^{G_{2k}}\(\psi \wr \lambda \)\right\rangle_{V_k^\tau} = \left\{\barr{ll} 1, &\text{if }\left\{\ba  &\epsilon_\tau(\psi) = 1\text{ and $\lambda$ has all even rows}, \\ &\epsilon_\tau(\psi)=-1\text{ and $\lambda$ has all even columns}, \ea\right. \\ \\ 0,&\text{otherwise}.\earr\right.\]
\end{lemma}

\begin{proof}
Fix $\pi \in W_k$.  The cycles of $\pi$ are either of the form $(2i_s-1, \ 2i_s)$ for $s=1,\dots,S$, or come in pairs of the form $(i_1^t\cyc i_2^t\cyc \cdots\ i_{\ell(t)}^t)$, $(j_1^t\cyc j_2^t\cyc \cdots\ j_{\ell(t)}^t)$ for $t=1,\dots, T$, where $(i_a^t\cyc j_a^t)$ is a cycle of $\omega_k$ for each $a,t$.  If $h \in \delta^\tau_k(H)$, then in the former case $h_{k+i_s} = \leftexp{\tau}{ h_{i_s}}$ and in the latter case $h_{j_a^t} = \leftexp{\tau}{h_{i_a^t}}$.  In addition, note that $\sgn(\pi) = (-1)^S$.  
By Lemma \ref{char-vals},
\[  \ba
 \sum_{h \in \delta^\tau_k(H)}  (\psi \wr \lambda)(h,\pi)
  = 
\chi^\lambda(\pi)  \sum_{h \in \delta^\tau_k(H)} \prod_{s=1}^S \psi\( h_{k+i_s} h_{i_s}\) \prod_{t=1}^T \psi\( h_{i_{\ell(t)}^t} \cdots h_{i_2^t}h_{i_1^t} \)\chi^\psi\( h_{j_{\ell(t)}^t} \cdots h_{j_2^t}h_{j_1^t} \) 
\\
=
\chi^\lambda(\pi) \prod_{s=1}^S \(\sum_{h\in H} \psi\( h\cdot \leftexp{\tau}h \)\) 
\prod_{t=1}^T\( \sum_{h_1,\dots,h_{\ell(t)} \in H} \psi\( h_1\cdots h_{\ell(t)} \) \psi\( \leftexp{\tau}{h_{1}} \cdots \leftexp{\tau}{h_{\ell(t)}}\) \).\ea\]
We have $\psi(h) = \psi(\leftexp{\tau}h^{-1})$ since $\epsilon_\tau(\psi) = \pm 1$.  Therefore 
\[\ba \sum_{h_1,\dots,h_{\ell(t)} \in H} \psi\( h_1\cdots h_{\ell(t)} \) \psi\( \leftexp{\tau}{h_{1}} \cdots \leftexp{\tau}{h_{\ell(t)}}\) &= |H|^{\ell(t)-1} \sum_{h \in H} \psi\( h\) \overline{\psi\( \leftexp{\tau}{h}^{-1} \)} \\&= |H|^{\ell(t)} \langle \psi, \psi \rangle_H = |H|^{\ell(t)}.\ea\]
Substituting this and $\epsilon_\tau(\psi) = \frac{1}{|H|} \sum_{h\in H} \psi\( h\cdot \leftexp{\tau}h\)$ into our expression above, and noting that $2S+\sum_{t=1}^T 2\ell(t) = 2k$, we obtain 
$\sum_{h \in \delta^\tau_k(H)} (\psi \wr \lambda) (h, \pi)
= |H|^k(\epsilon_\tau(\psi))^S  \chi^\lambda(\pi).
$ 
Since $\sgn(\pi) = (-1)^{S}$, applying this  identity gives
\[\ba \left\langle \One, \Res_{V_k^\tau}^{G_{2k}}\( \psi \wr \lambda \)\right\rangle_{V_k^\tau} 
&= \frac{1}{|V_k^\tau|} \sum_{\pi \in W_k} \sum_{h \in \delta_k^\tau(H)} (\psi \wr \lambda)(h,\pi) 
=\left\{\barr{ll} \left \langle \One, \Res_{W_k}^{S_{2k}}\( \chi^\lambda \) \right\rangle_{W_k}, &\text{if $\epsilon_\tau(\psi) = 1$}, \\ 
\left \langle \sgn, \Res_{W_k}^{S_{2k}}\( \chi^\lambda \)\right\rangle_{W_k},
&\text{if $\epsilon_\tau(\psi) = -1$}.\earr\right.\ea\]
Our result now follows from applying Frobenius reciprocity to Lemma \ref{IRS}. 
\end{proof}

Define another subgroup of $G_{2k}$ by
\[ I_k^\tau = \left\{ (h,(\pi,\pi))  \in G_{2k}: h = (h_1,\dots,h_k,\leftexp{\tau}h_1,\dots, \leftexp{\tau}h_k) \in H^{2k},\ \pi \in S_k\right\}\] where we view $(\pi,\pi) \in S_k \times S_k$ as an element of $S_{2k}$ in the obvious way.  We then have a second lemma.

\begin{lemma}\label{lem2}
Let $\psi$ be an irreducible character of $H$ with $\epsilon_\tau(\psi) =0$ and let $\lambda,\mu \in \sP({k})$.  
Define $\varpi_k \in S_{2k}\subset G_{2k}$ as the permutation given by 
\[ \ba \varpi_k(2i-1) &= i, \\
   \varpi_k(2i) &= i+k,\ea \qquad\text{for }i=1,\dots,k.\] Then $I_k^\tau = (G_k\times G_k) \cap \varpi_k^{-1}(V_k^\tau) \varpi_k$ and
\[ \left\langle\One , \Res_{I_k^\tau}^{G_k\times G_k}
\( (\psi \wr \lambda) \odot  (\overline{\leftexp{\tau}\psi}\wr \mu)\)
\right\rangle_{I_k^\tau} = \left\{\barr{ll}  1, &\text{if }\lambda = \mu, \\ 0,&\text{otherwise}.\earr\right.\] 

\end{lemma}

\begin{proof}
We first observe that if $\omega = (1\cyc k+1)(2\cyc k+2)\cdots(k\cyc 2k)  = \varpi_k\omega_k \varpi_k^{-1}\in S_{2k}$ then 
\[ \varpi_k^{-1}(V_k^\tau) \varpi_k = \left\{ (h,\pi) : \pi \in C_{S_{2k}}(\omega),\ h = (h_1,\dots,h_k,\leftexp{\tau}h_1,\dots, \leftexp{\tau}h_k) \in H^{2k}\right\}.\]  It immediately follows that $I_k^\tau = (G_k\times G_k) \cap \varpi_k^{-1}(V_k^\tau) \varpi_k$.  Next note that $\overline{\leftexp{\tau}\psi}(\leftexp{\tau}h)= \overline{\psi(h)}$ for $h \in H$ and that $\chi^\mu \in \Irr(S_n)$ is real-valued.  Hence   
\[ \(\overline{\leftexp{\tau}\psi} \wr \mu\) (\leftexp{\tau}h,\pi) = \overline{{(\psi\wr \mu)} (h,\pi)}\] for  $\pi \in S_k$ and $h \in H^k$, where we let $\leftexp{\tau}h = (\leftexp{\tau}h_1,\dots,\leftexp{\tau}h_k)$.  Therefore 
by Lemma \ref{char-vals} and an argument similar to the one used in the previous lemma, if $\pi \in S_k$ then we have
\[\ba  \sum_{h \in H^k} (\psi \wr \lambda)(h,\pi) \cdot ({\overline{\leftexp{\tau}\psi}}\wr \mu) (\leftexp{\tau}h,\pi)
&
=\sum_{h \in H^k} {(\psi\wr \lambda)}(h,\pi) \cdot \overline{{(\psi\wr \mu)} (h,\pi)} 
= |H|^k\chi^\lambda(\pi) \overline{\chi^\mu(\pi)}.
\ea\]
Our result now follows from
\[ \ba 
\left\langle\One , \Res_{I_k^\tau}^{G_k\times G_k}\((\psi \wr \lambda) \odot (\overline{\leftexp{\tau}\psi} \wr \mu)\)\right\rangle_{I_k^\tau} 
&
=
\frac{1}{|I_k^\tau|}\sum_{(\pi,h) \in I_k^\tau} \( (\psi\wr\lambda) \odot (\overline{\leftexp{\tau}\psi} \wr \mu)\) (h,\pi)
\\& =  \frac{1}{|I_k^\tau|} \sum_{\pi \in S_k} \sum_{h \in H^k} (\psi\wr \lambda)(h,\pi) \cdot ({\overline{\leftexp{\tau}\psi}}\wr \mu) (\leftexp{\tau}h,\pi)
= \langle \chi^\lambda,\chi^\mu\rangle_{S_k}.
\ea\]
\end{proof}

\newcommand\sR{\mathscr{R}}

We are now prepared to prove the following instrumental proposition.

\begin{proposition}\label{main-prop} The induction of the trivial character of $V_k^\tau$ to $G_{2k}$ decomposes as the multiplicity free sum
\[ \Ind^{G_{2k}}_{V_k^\tau}(\One) = \sum_{\theta} \chi_\theta,\]
where the sum is over all $\theta \in \sP_H(2k)$  such that for every irreducible character $\psi \in \Irr(H)$,
\begin{enumerate}
\item[(1)] $\theta(\psi) = \theta(\overline{\leftexp{\tau}\psi})$;
\item[(2)] $\theta(\psi)$ has all even columns if $\epsilon_\tau(\psi) = -1$;
\item[(3)] $\theta(\psi)$ has all even rows if $\epsilon_\tau(\psi) = 1$.
\end{enumerate}
\end{proposition}

This result generalizes Proposition 3 in \cite{B91-2}, which treats the case $\tau = 1$.  Our proof  derives from a pair of detailed but straightforward calculations using the preceding lemmas.  This approach differs somewhat from the inductive method used by Baddeley in \cite{B91-2}.

\begin{proof}
Choose $\theta \in \sP_H(2k)$ satisfying (1)-(3).  We first show that $\chi_\theta$ appears as a constituent of $\Ind^{G_{2k}}_{V_k^\tau}(\One)$ and then demonstrate that the given decomposition has the correct degree.
To this end,  define 
\[\eta_\theta=  \bigodot_{\psi \in \Irr(H)} \psi \wr  \theta(\psi),\qquad\text{so that}\qquad
\chi_\theta = \Ind_{S_\theta}^{G_{2k}}(\eta_\theta).\]
Let $s \in S_{2k}$ and define the subgroup $D_s = S_\theta \cap s^{-1}(V_k^\tau)s.$ 
Then by Frobenius reciprocity and Mackey's theorem, we have 
\[ \ba \left\langle \Ind_{V_k^\tau}^{G_{2k}}(\One), \chi_\theta\right\rangle_{G_{2k}}
&
=
\left\langle \Res^{G_{2k}}_{S_\theta} \( \Ind_{V_k^\tau}^{G_{2k}}\( \One\)\), \eta_\theta\right\rangle_{S_\theta}
&
\quad\text{(by Frobenius reciprocity),}
\\&
\geq
 \left \langle \Ind_{D_s}^{S_\theta} \( \One\), \eta_\theta\right\rangle_{S_\theta}
&
\quad\text{(by Mackey's theorem),}
\\&
=
\left \langle  \One, \Res_{D_s}^{S_\theta} (\eta_\theta)\right\rangle_{D_s}
&
\quad\text{(by Frobenius reciprocity).}
\ea\] 
%
%

Recall from Section \ref{prelim} that if $\psi \in \Irr(H)$ then the two irreducible characters $\psi, \overline{\leftexp{\tau}\psi}$ of $H$ are distinct if and only if $\epsilon_\tau(\psi) = 0$.  Therefore we can list the distinct elements of $\Irr(H)$ in the form \[\psi_1, \psi_1',\dots, \psi_r, \psi_r', \vartheta_1,\dots, \vartheta_s,\] where for all $i$ we have $\psi_i' = \overline{\leftexp{\tau}\psi_i}$ and $\epsilon_\tau(\psi_i) = \epsilon_\tau(\psi_i') = 0$ and $\epsilon_\tau(\vartheta_i) \neq 0$. 
Without loss of generality, we can assume that the products defining $\rho_\theta$ and $S_\theta$ proceed in the order of this list; a different ordering corresponds to a conjugate choice of $s$ in what follows.  Since $|\theta(\psi_i)| = |\theta(\psi_i')|$ and $|\theta(\vartheta_i)|$ is even for all $i$, if we define $s \in S_{2k}$ as the element 
\[ s = \(\varpi_{|\theta(\psi_1)|} , \dots, \varpi_{|\theta(\psi_r)|}, 1,\dots,1\) \in \prod_{i=1}^r S_{|\theta(\psi_i)|+ |\theta(\psi_i')|} \times \prod_{i=1}^s S_{|\theta(\vartheta_i)|} \subset S_{2k}\] where $\varpi_k$ for $k={|\theta(\psi_1)|}, \dots,{|\theta(\psi_r)|}$ is as in Lemma \ref{lem2}, then 
$D_s =\prod_{i=1}^r I_{|\theta(\psi_i)|}^\tau \times
 \prod_{i=1}^s V^\tau_{|\theta(\vartheta_i)|/2}.$
Consequently  $
 \langle  \One, \Res_{D_s}^{S_\theta} (\eta_\theta)\rangle_{D_s} = 
\varepsilon_0 \varepsilon_{\pm 1}$
where
\[ \ba \varepsilon_0 &=
\prod_{i=1}^r  \left\langle\One , \Res_{I_{|\theta(\psi_i)|}^\tau}^{G_{|\theta(\psi_i)|}\times G_{|\theta(\psi_i')|}}\(\(\psi_i \wr \theta(\psi_i)\) \odot \({\psi_i'} \wr \theta({\psi_i'}) \) \)\right\rangle_{I_{|\theta(\psi_i)|}^\tau}
,\\
\varepsilon_{\pm1} &= 
\prod_{i=1}^s 
 \left \langle  \One, \Res^{G_{|\theta(\vartheta_i)|}}_{V_{|\theta(\vartheta_i)|/2}^\tau} \(\vartheta_i \wr \theta(\vartheta_i) \)\right\rangle_{V_{|\theta(\vartheta_i)|/2}^\tau}
.
\ea
\] We have $\varepsilon_0 = 1$ by Lemma \ref{lem2} 
and $\varepsilon_{\pm1} = 1$ by Lemma \ref{lem1} 
 and so we conclude that if $\theta \in \sP_H(n)$ satisfies (1)-(3), then $\chi_\theta$ appears as a constituent of $\Ind_{V_k^\tau}^{G_n}(\One)$ with multiplicity at least one.   
 
 \def\cF{\mathcal{F}}
 
 To prove that this multiplicity is exactly one and that these are the only constituents, we show that both sides of the equation in the proposition statement have the same degree.  Define $\cF$ as the set  of functions $f: \Irr(H)\to \ZZ_{\geq 0}$ which have $f(\psi) =  |\theta(\psi)|$ for some $\theta \in \sP_H(2k)$ satisfying (1)-(3).  
Then 
the sum of the degrees of $\chi_\theta$ as $\theta \in \sP_H(2k)$ varies over all maps satisfying (1)-(3) is
  \[ \sum_\theta \deg(\chi_\theta) = \sum_\theta (2k)! \prod_{\psi \in \Irr(H)}  \frac{\deg(\psi)^{|\theta(\psi)|} \deg\(\chi^{\theta(\psi)}\)}{\theta(\psi)!} =
\sum_{f \in \cF} n!\hs \Pi_0(f)\hs \Pi_{\pm1}(f) \] 
where
\[\ba 
\Pi_0(f) &= 
\prod_{i=1}^r\( \sum_{\lambda \in \sP(f(\psi_i))} \( \frac{\deg(\psi_i)^{f(\psi_i)} \deg\(\chi^{\lambda}\)}{f({\psi_i})!}\)\(\frac{\deg({\psi_i'})^{f(\psi_i')} \deg\(\chi^{\lambda}\)}{f(\psi_i')!}\)\),
\ea\] and \[\ba
\Pi_{\pm1}(f) &= \prod_{\substack{\psi \in \Irr(H) \\ \epsilon_\tau(\psi) = -1}}\( \sum_{\substack{\lambda \in \sP(f(\psi))\text{ with}\\\text{all even columns}}} 
\frac{\deg(\psi)^{f(\psi)} \deg\(\chi^{\lambda}\)}{f(\psi)!}\)\prod_{\substack{\psi \in \Irr(H) \\ \epsilon_\tau(\psi) = 1}}\( \sum_{\substack{\lambda \in \sP(f(\psi))\text{ with}\\\text{all even rows}}} 
\frac{\deg(\psi)^{f(\psi)} \deg\(\chi^{\lambda}\)}{f(\psi)!}\) 
.
 \ea\]
Note that $\deg(\psi_i) = \deg(\psi_i')$ and $f({\psi_i}) = f(\psi_i')$ for all $i$ if $f \in \cF$.  Therefore 
\[ \Pi_0(f) 
=\prod_{i=1}^r
\frac{\deg(\psi_i)^{2f(\psi_i)} }{(f(\psi_i)!)^2} \(\sum_{\lambda \in \sP(f(\psi_i))} \deg\(\chi^\lambda\)^2\)
=\prod_{i=1}^r
\frac{\(2\deg(\psi_i)^2\)^{f(\psi_i)} }{2^{f(\psi_i)}f(\psi_i)!}.\]
Next, recall from Lemma \ref{IRS} that the sum $ \sum_\lambda \deg\(\chi^\lambda\) $ as $\lambda$ varies over the partitions of $2n$ with all even rows is  equal to $ \frac{(2n)!}{2^nn!}$, and that the sum over $\lambda$ with all even columns has the same value.  Thus
\[ \Pi_{\pm1} (f) = \prod_{i=1}^s\frac{\(\deg(\vartheta_i)^2\)^{f(\vartheta_i)/2}}{2^{f(\vartheta_i)/2} \(f(\vartheta_i)/2\)!}.\] 
As $f$ varies over all elements of $\cF$, the numbers $f(\psi_1),\dots,f(\psi_r), f(\vartheta_1)/2,\dots,f(\vartheta_s)/2$ range over all compositions of $k$.  Therefore, after substituting in the preceding expressions, we obtain by the multinomial formula 
  \[\ba\sum_\theta \deg(\chi_\theta)
&
=
\frac{ (2k)!}{2^kk!}  
\sum_{f \in \cF} k!
 \prod_{i=1}^r
\frac{\(2\deg(\psi_i)^2\)^{f(\psi_i)} }{f(\psi_i)!}
\prod_{i=1}^s
\frac{\(\deg(\vartheta_i)^2\)^{f(\vartheta_i)/2}}{(f(\vartheta_i)/2)!}
\\
&
=\frac{ (2k)!}{2^kk!}  \( \sum_{i=1}^r 2\deg(\psi_i)^2 + \sum_{i=1}^s \deg(\vartheta_i)^2 \)^k 
=\frac{ (2k)!}{2^kk!}  \( \sum_{\psi \in \Irr(H)} \deg(\psi)^2 \)^k 
 = \frac{|G_{2k}|}{|V_k^\tau|}
.
\ea
 \] Since this is precisely the degree of $\Ind_{V_k^\tau}^{G_n}(\One)$, the given decomposition 
   now follows by dimensional considerations.
\end{proof}

\subsection{Construction of a Model}\label{defns-section}

With this proposition in hand, we can now construct a generalized involution model for $G_n$ from any generalized involution model for $H$.  As above, we fix an automorphism $\tau \in \Aut(H)$ with $\tau^2 = 1$. 
Throughout this section, we assume there exists a model for $H$ given by a set of linear characters $\{ \lambda_i : H_i \to \CC \}_{i=1}^m$ for some positive integer $m$ and some subgroups $H_i\subset H$.

Our notation is intended to coincide with that of \cite{B91-2} when $\tau=1$.  Let $\sU_m$ denote the set of vectors $(x_0,x_1,\dots,x_m)$ with all entries nonnegative integers, and define \[ \sU_m(n) = \left\{ x \in \sU_m : 2x_0 + \sum_{i=1}^m x_i = n\right\}.\]  
Let $\sigma_k^\tau : V_k^\tau\rightarrow \{\pm1\}$ be the linear character given by  
\[ \sigma_k^\tau(h,\pi) = \sgn(\pi),\qquad\text{for }(h,\pi) \in V_k^\tau.\]  
For each $x \in \sU_m(n)$, we define a subgroup $G_x^\tau \subset G_n$ and a linear character $\phi_x^\tau: G_x^\tau \to \CC$ by
\be\label{grp-rep-defs} G_x^\tau = V_{x_0}^\tau \times \prod_{i=1}^m (H_i \wr S_{x_i})\qquad\text{and}\qquad \phi_x^\tau = \sigma_{x_0}^\tau \odot\bigodot_{i=1}^m \lambda_i \wr (x_i), \ee where on the right hand side $(x_i)$ denotes the trivial partition in $\sP(x_i)$ and we ignore terms corresponding to $i$ if $x_i=0$.

Given $x \in \sU_m(n)$, define 
\[\sR(x) = \left\{ \theta \in \sP_H(n) :  \barr{l}  \text{$x_i$ for each $i>0$ is the sum  of the number of odd columns in} \\\text{$\theta(\psi)$ as $\psi$ ranges over the irreducible constituents  of $\Ind_{H_i}^H (\lambda_i)$}
\earr
\right\}.\]
We then have the following extension of Theorem 1 in \cite{B91-2}, which treats the special case $\tau = 1$.  
  

\begin{theorem}\label{thm1}
Suppose 
$\epsilon_\tau(\psi) = 1$ for every irreducible character $\psi$ of $H$.  Then 
\[\Ind_{G_x^\tau}^{G_{n}} ( \phi_x^\tau) = \sum_{\theta \in \sR(x)} \chi_\theta,\qquad\text{for }x\in \sU_m(n),\] and $\{\phi_x^\tau : G_x^\tau \to \CC\}_{x \in \sU_m(n)}$ is a
model for $G_n = H \wr S_n$.
\end{theorem} 

The proof of this is in principle the same as that of \cite[Theorem 1]{B91-2} with all references to Baddeley's Proposition 3 replaced by ones to our Proposition \ref{main-prop}.  This does not quite work in practice, however, since Baddeley's proof in \cite{B91-2} makes no mention of Proposition 3 and instead uses two intermediate results which we have sidestepped.  For completeness we therefore give the following proof.

\begin{proof}
By the transitivity of induction we have
\be\label{1} \ba \Ind_{G_x^\tau}^{G_{n}} ( \phi_x^\tau) &= 
\Ind_{G_{2x_0} \times G_{x_1}\times \cdots \times G_{x_m}}^{G_n}\(\Ind_{V_{x_0}^\tau}^{G_{2x_0}} (\sigma_{x_0}^\tau) \odot \bigodot_{i=1}^m \Ind_{H_i \wr S_{x_i}}^{G_{x_i}} \(\lambda_i \wr (x_i)\)\).
\ea\ee  
Note that if $\theta \in \sP_H(n)$, then $\chi_\theta \otimes \widetilde\sgn = \chi_{\theta'}$ where $\theta' \in \sP_H(n)$ is defined by setting $\theta' (\psi)$ equal to the transpose of $\theta(\psi)$.  Therefore,  
since $\epsilon_\tau(\psi) = 1$ for all $\psi \in \Irr(H)$, we have by Proposition \ref{main-prop} that
$\Ind_{V_k^\tau}^{G_{2k}} (\sigma_k^\tau) = \Ind_{V_k^\tau}^{G_{2k}} (\One) \otimes \widetilde \sgn = \sum_\theta \rho_\theta$ 
 where the sum ranges over all $\theta \in \sP_H(n)$ such that $\theta(\psi)$ has all even columns for all $\psi \in \Irr(H)$. 
 Also, Proposition 1 in \cite{B91-2} states that 
$\Ind_{ H_i \wr S_{x_i}}^{G_{x_i}}\( \lambda_i \wr (x_i) \) = \sum_{\theta} \chi_\theta$ 
where the sum is over all $\theta \in \sP_H(x_i)$ such that $\theta(\psi)$ is the zero partition if $\psi$ is not a constituent of $\Ind_{H_i}^H(\lambda_i)$ and a trivial partition otherwise.

Given these facts, we can completely decompose $\Ind_{G_x^\tau}^{G_{n}} ( \phi_x^\tau)$ by using Lemma 1 in \cite{B91-2}, which shows that if $\psi$ is a representation of $H$ and $\alpha \in \sP(a)$ and $\beta \in \sP(b)$, then 
$\Ind_{G_a\times G_b}^{G_{a+b}}\(  (\psi \wr \alpha)\odot (\psi \wr \beta) \) = \sum_{\gamma \in \sP(a+b)} c_{\alpha,\beta}^\gamma (\psi \wr \gamma)$ where the coefficients $c_{\alpha,\beta}^\gamma$ are the nonnegative integers afforded by the Littlewood-Richardson rule.  
 Thus, after applying our substitutions to (\ref{1}) we can invoke Young's rule to obtain the desired decomposition.  
\end{proof}

The automorphism $\tau \in \Aut(H)$ naturally extends to an automorphism of $H^n$ and of $G_n$ via the definitions
\be\label{extend} \ba \leftexp{\tau}(h_1,\dots,h_n) & \overset{\mathrm{def}}= (\leftexp{\tau}h_1,\dots, \leftexp{\tau}h_n), &&\qquad\text{for }(h_1,\dots,h_n) \in H^n, \\ 
\leftexp{\tau}(h,\pi) &\overset{\mathrm{def}}= (\leftexp{\tau}h,\pi),&&\qquad \text{for }\pi \in S_n,\ h\in H^n
.
\ea\ee As in (\ref{omega_k}), let 
$\omega_k =(1\cyc 2)(3\cyc 4)\cdots (2k-1\cyc 2k) \in S_{2k},$ where by convention $\omega_0 = 1$.
We now have the following generalization of Theorem 2 in \cite{B91-2}.  

\begin{theorem}\label{main-thm} 
Suppose $\{\lambda_i : H_i\to \CC\}_{i=1}^m$  is a generalized involution model for $H$ with respect to $\tau \in \Aut(H)$, so that there exists a set $\{\varepsilon_i\}_{i=1}^m$ of orbit representatives in $\cI_{H,\tau}$ with $H_i = C_{H,\tau}(\varepsilon_i)$. For each $x \in \sU_m(n)$, define 
\[ \varepsilon_x = \((\underbrace{1,\dots, 1}_{2x_0\text{ times}}, \underbrace{\varepsilon_1,\dots, \varepsilon_1}_{x_1\text{ times}}, \underbrace{\varepsilon_2,\dots,\varepsilon_2}_{x_2\text{ times}},\dots, \underbrace{\varepsilon_m,\dots,\varepsilon_m}_{x_m\text{ times}}) , \omega_{x_0}\) \in G_n.\]
If we extend $\tau$ to an automorphism of $G_n$ by (\ref{extend}), then the linear characters
$\{ \phi_x^\tau : G_x^\tau \rightarrow \CC \}_{x \in \sU_m(x)}$ form a generalized involution model for $G_n$ with respect to $\tau$.

%
%
%
 
\end{theorem}


\begin{proof}
By Theorem \ref{bg-thm}, we have $\epsilon_\tau(\psi) = 1$ for all $\psi \in \Irr(H)$.  Since $\{\lambda_i\}_{i=1}^m$ is a model for $H$, it follows from Theorem \ref{thm1} that $\{ \phi_x^\tau \}_{x \in \sU_m(x)}$ is a model for $G_n$.  To show this model is a generalized involution model, we must prove both of the following:
\begin{enumerate}
\item[(1)] For each $x \in \sU_m(n)$, the group $G_x^\tau$ is the $\tau$-twisted centralizer in $G_n$ of $\varepsilon_x \in \cI_{G_n,\tau}$.

\item[(2)] The set 
$\{\varepsilon_x\}_{x \in \sU_m(n)}$ contains exactly one element from each orbit in $\cI_{G_n,\tau}$.

\end{enumerate}
To this end, fix $x \in \sU_m(n)$ and let $\varepsilon'_x \in H^n$ be the element with $\varepsilon_x = \( \varepsilon_x', \omega_{x_0}\)$.  Since $\varepsilon'_x \cdot  \leftexp{\tau} \varepsilon'_x = 1 \in H^n$ and $\omega_{x_0}(\varepsilon_x') = \varepsilon_x'$ by assumption, we have 
$\varepsilon_x \in \cI_{G_n,\tau}$.  

Next, let $\pi \in S_n$ and $h \in H^n$ and consider the twisted conjugation of $\varepsilon_x$ by the arbitrary element $g = \(\pi^{-1}(h),\pi\) \in G_n$.  This gives
\be \label{id1}(k,\sigma) \overset{\mathrm{def}}=g \cdot \varepsilon_x \cdot  \leftexp{\tau}{g}^{-1} =  \( \pi\omega_{x_0}\pi^{-1}(h) \cdot  \pi(\varepsilon'_x)\cdot  \leftexp{\tau}{h}^{-1},\pi \omega_{x_0}\pi^{-1}\).\ee  Hence $g \in C_{G_n,\tau}(\omega_x)$ only if $\pi \in C_{S_n}(\omega_{x_0})$.  Assume this, and define $J_0 = \{1,\dots,2x_0\}$ and $ J_i= \left\{ 2x_0+\(x_1+\dots +x_{i-1}\) + j : 1\leq j \leq x_i\right\}$ for $i=1,\dots,m$.
Then  $\pi$ permutes the sets $J_0$ and $J_1 \cup \dots \cup J_m$, so 
\be\label{identity} k_j = \left\{\ba &h_{j'}\cdot \leftexp{\tau}h_{j}^{-1}, &&\text{if } j \in J_0,\text{ where $j'=\omega_{x_0}(j)$}, \\ 
&h_j \cdot \varepsilon_{i} \cdot \leftexp{\tau}h_j^{-1}, && \text{otherwise, where $i$ is the unique index with $\pi^{-1}(j) \in J_i$}.
\ea \right.\ee  It follows from the first case in this identity that $k=\varepsilon'_x$ only if $h_{j'} = \leftexp{\tau} h_{j}$ for all $j \in J_0$.  It follows from the second case that if $j \in J_1\cup \dots \cup J_m$ then $k_j$ lies in the $H$-orbit of $\varepsilon_i$, where $i$ is the unique index with $\pi^{-1}(j) \in J_i$.  
Thus, $k = \varepsilon'_x$ only if  $\pi$ also permutes each of the sets $J_i$ and  $h_j \in C_{H,\tau}(\varepsilon_i) = H_i$ for all $j \in J_i$ and $i=1,\dots,m$.  Combining these observations, we see that $g \in C_{G_n,\tau}(\varepsilon_x)$ only if $g \in G_x^\tau$.  The reverse implication follows easily, and so we have $C_{G_n,\tau}(\varepsilon_x) = G_x^\tau$.  

It remains to show that the elements $\varepsilon_x$ for $x \in \sU_m(n)$ represent the distinct $\tau$-twisted conjugacy classes in $\cI_{G_n,\tau}$.  This requires a straightforward but tedious calculation, similar to the one in the previous paragraph.  We leave this to the reader.
\end{proof}



We conclude this section with an observation on how to construct a Gelfand model for $G_n$ from a generalized involution model for $H$.  To make our notation more concise, we adopt the following convention: given $g = (h,\pi) \in G_n$, define $|g| \in S_n$ and $z_g: \{1,\dots,n\} \to H$ by
\be\label{wreath-defs} |g| = \pi \in S_n\qquad\text{and}\qquad z_g(i) = h_{i} \in H.\ee 
We can identify $G_n$ with the set of $n\times n$ matrices which have exactly one nonzero entry in each row and column, and whose nonzero entries are elements of $H$.
 Viewing $g \in G_n$ as a matrix of this form, $|g|$ is the matrix given by replacing each nonzero entry of $g$ with $1$, and $z_g(i)$ is the value of the nonzero entry of the matrix $g$ in the $i$th column. 

In the following statement, it helps to recall the definition of $\sign_{S_n}$ from (\ref{sign_S_n}).  The symbol $\tau$ continues to denote a fixed automorphism of $H$ with $\tau^2=1$, which we have extended to an automorphism of $G_n$ by (\ref{extend}).  Also, $\KK$ here denotes a fixed subfield of $\CC$ and $\cV_{H,\tau}$, $\cV_{G,\tau}$ are the vector spaces over $\KK$ defined by (\ref{cV}).

\begin{proposition} \label{observation2} Suppose $\sign_H : H\times \cI_{H,\tau} \to \KK$ is a function such that the map $\rho : H\to \GL(\cV_{H,\tau})$ defined by
\[ \rho(h) C_\omega = \sign_H(h,\omega)\cdot C_{h\cdot\omega\cdot \leftexp{\tau}h^{-1}},\qquad\text{for }h\in H,\ \omega \in \cI_{H,\tau}\] is a Gelfand model for $H$.  Then the map $\rho_{n,H} : G_n \to \GL(\cV_{G,\tau})$ defined by
\[ \rho_{n,H}(g) C_{\omega} =\sign_{G_n}(g,\omega)  \cdot C_{g\cdot\omega \cdot \leftexp{\tau}g^{-1}},\qquad\text{for }g\in G_n,\ \omega \in \cI_{G_n,\tau},\] where
\[ \sign_{G_n}(g,\omega) =   
\sign_{S_n}(|g|,|\omega|) \prod_{i \in \Fix(|\omega|)} \sign_H(z_g(i), z_\omega(i))\] is a Gelfand model for $G_n = H \wr S_n$.
\end{proposition}

\begin{proof}
By Lemma \ref{observation}, $H$ possesses a generalized involution model $\{\lambda_i :H_i \to \KK\}_{i=1}^m$ with respect to $\tau$.  Retaining the notation of Theorem \ref{main-thm}, we may assume without loss of generality that $\lambda_i(h) = \sign_H(h, \varepsilon_i)$ for all $h \in C_{H,\tau}(\varepsilon_i) = H_i$ for each $i=1,\dots,m$.  To prove that $\rho_{n,H}$ is a Gelfand model, it suffices by Lemma \ref{observation} to show only two things: that $\phi_x^\tau(g) = \sign_{G_n}(g,\varepsilon_x)$ for all $g \in G_x^\tau$ for each $x \in \sU_m(n)$, and that $\rho_{n,H}$ is a representation.

To this end, fix $x \in \sU_m(n)$ and consider $g \in  H_i \wr S_{x_i} $.  Since $\lambda_i$ is a linear character, we have by Lemma \ref{char-vals} that
\[  (\lambda_i \wr (x_i))(g) ={\displaystyle\prod_{j=1}^{x_i} } \lambda_i(z_g(j)) 
= {\displaystyle\prod_{j=1}^{x_i} } \sign_H(z_g(j), \varepsilon_i).
\] Thus if $g = (g_0,g_1,\dots,g_m) \in G_x^\tau$, where $g_0 \in V_{x_0}^\tau$ and $g_i \in H_i \wr S_{x_i} $ for $i=1,\dots,m$, then
\[\ba  \phi_x^\tau(g)& = \sigma_{x_0}^\tau(g_0) \prod_{i=1}^m  (\lambda_i \wr x_i)(g_i)
\\
&
=
 \sign_{S_n}(|g|,|\varepsilon_x|) \prod_{i\in \Fix(\omega_x)} \sign_H(z_g(i), z_{\omega_x}(i)) = \sign_{G_n}(g,\varepsilon_x).\ea\]

It remains to show that
$\rho_{n,A}$ is a representation. 
Let $g,h\in G_n$ and $\omega \in \cI_{G_n,\tau}$ and write $\omega' = h\cdot \omega\cdot \leftexp{\tau}h^{-1}$.  First, by Lemma \ref{APR} we have
\be\label{tag1}  \sign_{S_n}(|g|, |\omega'|) \cdot \sign_{S_n}(|h|, |\omega|) = \sign_{S_n}(|gh|, |\omega|).\ee 
Now let $\pi = |h|$.  Choose $i \in \Fix(|\omega|)$ and observe that $\pi(i) \in \Fix(|\omega'|)$.  It follows from the fact that $\omega\cdot\leftexp{\tau}\omega = \omega' \cdot\leftexp{\tau}\omega' = 1$ that both $z_\omega(i)$ and $z_{\omega'}\circ\pi(i)$ belong to $\cI_{H,\tau}$.  Furthermore, one can check that
\[ z_g\circ\pi(i) \cdot z_h(i) = z_{gh}(i)\qquad\text{and}\qquad z_{\omega'}\circ \pi(i) =  z_h(i) \cdot z_\omega(i) \cdot \leftexp{\tau}z_h(i)^{-1}
.\]  Since $\sign_H(a,b\cdot x\cdot \leftexp{\tau}b^{-1})\cdot\sign_H(b,x)= \sign_H(ab,x)$ for $a,b \in H$ and $x \in \cI_{H,\tau}$, it follows that
 \be\label{tag2} 
 \ba \sign_H(z_g\circ \pi(i), z_{\omega'}\circ \pi(i)) \cdot \sign_H(z_h (i), z_{\omega}(i)) 
 = \sign_H(z_{gh}(i), z_{\omega}(i)).\ea\ee   Since $\Fix(|\omega'|) = \{ \pi(i) : i \in \Fix(|\omega|)\}$, combining the identities (\ref{tag1}) and (\ref{tag2}) shows that $\sign_{G_n}(g,\omega') \cdot \sign_{G_n}(h,\omega) = \sign_{G_n}(gh,\omega)$, which suffices to show that 
 $\rho_{n,H}$ is a representation, and therefore a Gelfand model.
\end{proof}

\section{Applications}\label{application}

As an application of Theorem \ref{main-thm}, we construct in this section a generalized involution model and a Gelfand model for $G_n = H\wr S_n$ when $H$ is abelian.  This gives a simple proof of Theorem 1.2 in \cite{APR2008}, which asserts that the representation $\rho_{r,n}$ from the introduction is a Gelfand model for $G_n$ in the special case that $H$ is the cyclic group of order $r$.  Using Theorem \ref{thm1}, we prove some facts concerning the decomposition of this representation into irreducible constituents, and in so doing prove a conjecture of Adin, Postnikov, and Roichman from \cite{APR2008}. 


Throughout this section, let $A$ be a finite abelian group and let $\tau \in \Aut(A)$ be the automorphism defined by $\leftexp{\tau} a = a^{-1}$.  For this particular case, we note that 
\[ \ba &\cI_A \overset{\mathrm{}}= \{ a \in A : a^2 = 1\}, \\ 
&\cI_{A,\tau}  \overset{\mathrm{}}= \{ a \in A : a \cdot \leftexp{\tau}a  = 1\} = A, 
\ea 
\qquad
\ba &C_{A}(a)  \overset{\mathrm{}} = \{b \in A: bab^{-1} = a\} = A,\\
&C_{A,\tau}(a)   \overset{\mathrm{}} = \{b \in A: b\cdot a\cdot \leftexp{\tau}b^{-1} = a\} = \cI_A.\ea 
\]
The automorphism $\tau$ gives rise to the following generalized involution model for $A$.


\begin{lemma}\label{abelian-model-lem}
If $A$ is abelian, then the set $\Irr(\cI_A)$ of all irreducible characters of the subgroup $\cI_A = \{ a \in A : a^2=1\}$  forms a generalized involution model for $A$ with respect to the automorphism $\tau : a\mapsto a^{-1}$.  In particular, for each $\lambda \in \Irr(\cI_A)$, the induced character
$ \Ind_{\cI_A}^{A} (\lambda)$ is the sum of all $\psi \in \Irr(A)$ with $\Res_{\cI_A}^{A}(\psi) = \lambda$.
\end{lemma}

\begin{remark} This generalized involution model is clearly unique, up to the arbitrary assignment of irreducible representations of $\cI_A$ to orbits in $\cI_{A,\tau}$, since the degree of any Gelfand model for $A$ is $|A|$ and so we must have $\cI_{A,\tau} = A$. 
\end{remark}

\begin{proof}
Since $\cI_{A,\tau} = A$ and $\cI_A = C_{A,\tau}(a)$ for every $a \in A$, there are $|\cI_A|$ distinct twisted conjugacy classes in $\cI_{A,\tau}$ and so
each irreducible character of $\cI_A$ can be viewed as  a linear character of the $\tau$-twisted centralizer of a representative of a distinct orbit in $\cI_{A,\tau}$.  The claimed decomposition of $ \Ind_{\cI_A}^{A} (\lambda)$ is immediate by Frobenius reciprocity, and since each element of $\Irr(A)$ restricts to an element of $\Irr(\cI_A)$, our assertion follows. %
%
%
\end{proof}

\def\sorb{s_{\mathrm{orb}}}
\def\sinv{s_{}}

Seeing this result, we naturally want to use Proposition \ref{observation2} to obtain a Gelfand model for the wreath product $A\wr S_n$.  In order to do this, we must first define a function $\sign_A : A\times A \to \CC$ which corresponds to the generalized involution model for $A$ just described.
We will define this function in two different ways: first from a completely abstract standpoint which does depend on the structure of $A$, and then with an explicit construction which relies on a given decomposition of $A$ as a direct product of cyclic groups.

For our first definition, we must introduce a few pieces of notation to keep track of our arbitrary but unspecified sets of orbit representatives.
Let $B = \{ a^2 : a \in A\}$ and observe that the cosets of this subgroup in $A$ are precisely the  orbits in $\cI_{A,\tau}$ under the twisted conjugacy action $a: x \mapsto a\cdot x\cdot \leftexp{\tau}a^{-1} = a^2x$.  Fix a bijection between $A/B$ and  $\Irr(\cI_A)$, and for each $x \in A$, let $\lambda_x : \cI_A \to \CC$ denote the linear character corresponding to the orbit $x B$.
Now choose two maps 
\[ \wt s_{\mathrm{orb}} : A/B \to A\qquad\text{and}\qquad \wt\sinv : A/\cI_A \to A\]
assigning representatives to the cosets of $B$ and $\cI_A$  in $A$, and let
\[ \sorb(a) = \wt s_{\mathrm{orb}}(a B)\qquad\text{and}\qquad    \sinv(a) = \wt\sinv(a\cI_A),\qquad\text{for }a\in A.\]  
The image of $\sorb$ is then a set of orbit representatives in $A$, which explains our notation.
Our next definition is our most complicated: let  $q : A \to A$ be the map
\[ q(a)  = \wt\sinv\(\left\{ b \in A :  \sorb(a)\cdot b^2= a\right\}\),\qquad\text{for }a\in A,\]
 The set $\left\{ b \in A :   \sorb(a)\cdot b^2= a\right\}$ is a coset of $\cI_A$ in $A$ and so the map $q$ is well-defined.  We can think of the value of $q(a)$ as the square root of $a$ modulo $B$. In the case that $A$ is cyclic, $q$ has a much more direct formula which we will compute.  

We now define $\sign_A : A\times A\to \CC$ as the function 
\be\label{sign} \sign_A(a,x) = \lambda_x\(a \cdot q(x) \cdot \sinv\(a\cdot q(x)\)^{-1}\) \ee and let $\rho_A : A \to \GL(\cV_{A,\tau})$  be the map given by
\be \rho_A(a) C_x =\sign_A(a,x) \cdot C_{a^2x},\qquad\text{for }a,x \in A.\ee 
%
These definitions come with the following result.

\begin{proposition}\label{abelian-model}
The map $\rho_A$ defines a Gelfand model for the abelian group $A$.
\end{proposition}

\begin{proof}
If $a \in \cI_A$, then $\sinv(a\cdot q(x)) =  \sinv(q(x))=q(x)$ and so $\sign_A(a,x) =\lambda_x(a)$.  Therefore, by Lemma  \ref{observation} and the preceding lemma, it suffices to show that $\rho_A$ is a representation.  For this, fix $a,b,x \in A$ and observe that $q(b^2x) = \sinv(b \cdot q(x))$ since
\[  \sorb(x) \cdot (b\cdot q(x))^2 =b^2 \cdot  \sorb(x)\cdot q(x)^2 = b^2x.\]
In addition, since $\sinv(c) \cI_A = c \cI_A$ for all $c \in A$, we have $\sinv\( a\cdot \sinv(b\cdot q(x))\) = \sinv\( ab\cdot q(x)\)$. 
Thus, since  $\lambda_x =\lambda_{b^2x}$ by construction, $ \sign_A(a,b^2x) =  \lambda_x\(a \cdot \sinv(b\cdot q(x)) \cdot \sinv\(ab\cdot q(x)\)^{-1}\)$
 and so $\sign_A(b,x) \cdot \sign_A(a,b^2x) = \sign_A(ab,x)$, which suffices to show that 
$\rho_A$ is a representation.
\end{proof}

Using this abstract formulation, we can provide a concrete definition of $\sign_A$ using the structure of $A$ as a finite abelian group.   For any two integers $a\leq b$, let $[a,b] = \{ i \in \ZZ : a\leq i\leq b\}$.  Identify  the cyclic group $\ZZ_r$ with the set $[0,r-1]$ so that the group operation is addition modulo $r$, and define a function  $\sign_r : \ZZ_r \times \ZZ_r \to \{\pm 1\}$ by
\[
\sign_{r}(a,x) = \left\{\barr{rl} -1,&\text{if $r$ is even and there exists $k \in [0,r/2-1]$} \\ &\text{with $x = 2k+1$ and $a+k \in[r/2,r-1]$,} \\ \\ 1,&\text{otherwise},\earr\right.\qquad\text{for }a,b \in \ZZ_r.
\]
If $A =\prod_{i=1}^k \ZZ_{r_i}$ where each $r_i$ is a prime power, then we define $\sign_A: A\times A\to \{\pm1\}$ by 
\be\label{general}\sign_A(a,x) = \prod_{i=1}^k \sign_{r_i}(a_i,x_i),\qquad\text{for }a=(a_1,\dots,a_k)\in A \text{, }x = (x_1,\dots,x_k) \in A.\ee
 Every finite abelian group is isomorphic to a direct product of this form which is unique up to rearrangement of factors, so the formula (\ref{general}) is well-defined for all abelian groups.  The definition (\ref{general}) is just a special case of (\ref{sign}), which explains the following corollary.
 
 \begin{corollary} \label{abelian-model-cor}
 If $A$ is abelian then the map $\rho_A$ with $\sign_A$ defined by (\ref{general}) is a Gelfand model.
\end{corollary}

\begin{proof}
It suffices to prove this when $A = \ZZ_r$ is cyclic, for this we only need to show that $\sign_A = \sign_r$ for some choice of the sections  $s_{\mathrm{orb}}$ and $s$ and of the arbitrary correspondence between orbits in $\cI_{A,\tau}$ and irreducible representations of $\cI_A$.   
If $r$ is odd then this always happens since $\cI_A = \{1\}$ so $\sign_A(a,x) = \sign_r(a,x) = 1$ for all $a,x \in A$.  Suppose $r$ is even.  
Then $\cI_A = \{0, r/2\}$; the cosets $A/\cI_A$ are $[0,r/2-1]$ and $[r/2,r-1]$; and the two orbits in $\cI_{A,\tau} =A$ are given by the sets of odd and even integers in $[0,r-1]$.  Assign the trivial representation of $\cI_A$ to the even orbit and the nontrivial representation to the odd orbit, so that the notation $\lambda_x : \cI_A \to \CC$ becomes
\[ \lambda_x(0) = 1\quad\text{and}\quad \lambda_x(r/2) = \left\{\barr{rl} 1, & \text{if $x$ is even}, \\ -1, &\text{if $x$ is odd},\earr\right.\qquad\text{for }x \in A.\]
If we define the sections $\sorb$ and $\sinv$ by 
\[ \sorb(a) = \left\{\barr{ll} 0,&\text{if $a$ is even}, \\ 1,&\text{if $a$ is odd},\earr\right.
\quad
\text{and}
\quad
\sinv(a) = \left\{ \barr{ll} a, & \text{if }a \in[0,r/2-1], \\ a-r/2,&\text{if }a \in [r/2,r-1],\earr\right.\] then the function $q: A\to A$ is given by the simple formula $q(a) = \lfloor a/2\rfloor$ for $a\in A$, where the floor function takes its usual meaning for integers.  It now follows by inspection that with respect to these choices, the definition  (\ref{sign})  of $\sign_A$  matches $\sign_r$ as required.
\end{proof}

We are now in a position to apply Proposition \ref{observation2} to obtain a Gelfand model for the wreath product $G_n = A \wr S_n$.  In particular, extend $\tau$ to an automorphism $\tau \in \Aut(G_n)$ by $\leftexp{\tau}(a,\pi) = (a^{-1},\pi)$, and define a map $\rho_{n,A} : G_n \to \GL(\cV_{G_n,\tau})$ by
\[ \rho_{n,A}(g) C_{\omega} =\sign_{G_n}(g,\omega)  \cdot C_{g\cdot\omega \cdot \leftexp{\tau}g^{-1}},\qquad\text{for }g\in G_n,\ \omega \in \cI_{G_n,\tau},\] where
\[ \sign_{G_n}(g,\omega) =   
\sign_{S_n}(|g|,|\omega|) \prod_{i \in \Fix(|\omega|)} \sign_A(z_g(i), z_\omega(i)).\] Here $\sign_{S_n}$ is given by (\ref{sign_S_n}) and $\sign_A$ is given by either (\ref{sign}) or (\ref{general}).  
The following theorem is now immediate from Proposition \ref{observation2} and the preceding two results.

\begin{theorem}
The map $\rho_{n,A}$ defines a Gelfand model for $G_n = A\wr S_n$ when $A$ is abelian.
\end{theorem}

By restating this theorem in slightly greater detail in the special case that $A$ is cyclic, we can explain the formula (\ref{intro}) from the introduction and provide an alternate proof of Theorem 1.2 in \cite{APR2008}.  For this, we view $\ZZ_r$ as the additive group of integers $[0,r-1]$, so that 
\be\label{mult-ref} (a,\pi)(b,\sigma) = (\sigma^{-1}(a) + b, \pi \sigma),\qquad\text{for }(a,\pi), (b,\sigma) \in \ZZ_r \wr S_n.\ee  We let $(a,\pi)^T = (-a,\pi)^{-1} = \(\pi(a),\pi^{-1}\)$ for $(a,\pi) \in \ZZ_r \wr S_n$ and define
\[\cV_{r,n} = \QQ\spanning\left\{ C_\omega : \omega \in \ZZ_r \wr S_n,\ \omega^T = \omega\right\}.\] Observe that $g^T = \leftexp{\tau} g^{-1}$ for $g \in \ZZ_r \wr S_n$, where $\tau$ is the automorphism $\leftexp{\tau}(a,\pi) = (-a,\pi)$.  Therefore $\cV_{r,n} = \cV_{G, \tau}$ with $G=\ZZ_r \wr S_n$ in our earlier notation.  Also, if we view elements of the wreath product $\ZZ_r \wr S_n$ as generalized permutation matrices, then $g^T$ is to the usual matrix transpose of $g$. As element $g \in \ZZ_r \wr S_n$ is \emph{symmetric} or an \emph{absolute involution} if $g^T = g$. 

Recall the definition of $|g|$ and $z_g$ for $g \in \ZZ_r \wr S_n$ from (\ref{wreath-defs}).  
The following notation comes from Definitions 6.1 and 6.3 in \cite{APR2008}.   For $g,\omega \in \ZZ_r \wr S_n$, let $B(g,\omega)$ denote the subset of $\{1,\dots,n\}$ given by 
\[ B(g,\omega) =\left\{\barr{ll} \varnothing, &\text{if $r$ is odd}, \\  \\
\left\{ i \in \Fix(|\omega|) 
: \ba 
&z_\omega(i) \text{ is odd and } z_g(i)+k\in[r/2,r-1] \\
& \text{for the $k \in [0,r/2-1]$ with $2k+1=z_\omega(i)$}\ea
\right\}, &\text{if $r$ is even.}\earr\right.\] 
Next define
\[ \sign_{r,n}(g,\omega) = (-1)^{|B(g,\omega)|}\cdot (-1)^{|\Inv(|g|) \cap \Pair(|\omega|)|}\]
and let  $\rho_{r,n} : \ZZ_r \wr S_n \to \GL(\cV_{r,n})$ be the map given by
\[ \rho_{r,n}(g) C_\omega ={\sign_{r,n}(g, \omega)}_{} \cdot C_{g \omega g^T},\qquad\text{for }g,\omega \in \ZZ_r \wr S_n \text{ with }\omega^T = \omega.\] 
The map $\rho_{r,n}$ is precisely the representation $\rho_{n,A}$ above with $A = \ZZ_r$ and $\sign_A = \sign_{r}$, and one can check that our definition of $\sign_{r,n}$ agrees with the one given on generators in the introduction.  We thus  obtain the following corollary, which appears as Theorem 1.2 in \cite{APR2008}.

\begin{corollary} \label{thm1.2} (Adin, Postnikov, Roichman \cite{APR2008}) The map $\rho_{r,n}$ defines a Gelfand model for the wreath product $\ZZ_r \wr S_n$.
\end{corollary}

By directly applying Theorem \ref{main-thm} to Lemma \ref{abelian-model-lem}, we can explicitly describe the generalized involution model for $\ZZ_r \wr S_n$ whose existence is implicit in our construction of $\rho_{r,n}$.   
In this situation, it is convenient to identify $\ZZ_r$ with the multiplicative subgroup of $\CC^\times$ given by all $r$th roots of unity; thus $\ZZ_2 = \{\pm 1\}$.  Let $\zeta_r = e^{2\pi i / r}$ be a primitive $r$th root of unity.  We view $\ZZ_r \wr S_n$ as the multiplicative group of $n\times n$ generalized permutation matrices whose nonzero entries are taken from $\ZZ_r$.  Given $g \in \ZZ_r \wr S_n$, let $|g|$ denote the permutation matrix given by replacing each entry of $g$ with its absolute value, and let $z_g(i)$ for $i=1,\dots,n$ denote the nonzero entry of $g$ in its $i$th column.  Under our previous conventions, the matrix $g$ can then be identified with the abstract pair $(x,\pi)$ where $\pi = |g| \in S_n$ and $x_i = z_g(i) \in \ZZ_r$ for $i=1,\dots,n$.  The matrix transpose $g^T$ then coincides with our previous definition of the transpose.

 For each $i \in [0,r-1]$, let $\psi_i : \ZZ_r \to \CC$ denote the irreducible character 
\[ \psi_i(x) = x^i,\qquad\text{for $x\in \ZZ_r$ viewed as an element of $\CC^\times$}\] so that $\Irr(\ZZ_r) = \{ \psi_i :i \in [0,r-1]\}$.  
Additionally let \[ \ba \sP &= \text{the set of all partitions of nonnegative integers,} \\
 \sP_r(n) &=\text{the set of $r$-tuples $\theta = (\theta_0,\theta_1,\dots,\theta_{r-1})$ of partitions with $|\theta_0|+|\theta_1|+\dots+|\theta_{r-1}| = n$}.\ea\] 
We refer to  elements of $\sP_r(n)$ as \emph{$r$-partite partitions of $n$}.  Define  $\psi_i \wr \lambda$ for $i \in [0,r-1]$ and  $\lambda \in \sP$  as the character of $\ZZ_r \wr S_{|\lambda|}$ given by 
 \[ \(\psi_i \wr \lambda\) (g) = \chi^\lambda(|g|)\( \frac{\det (g)}{\det(|g|)}\)^i,\qquad\text{for }g \in \ZZ_r \wr S_{|\lambda|}.\]
One checks via Lemma \ref{char-vals} that this coincides with our constructions in Section \ref{wreath-chars} since $\ZZ_r$ is abelian and since ${\det (g)}/{\det(|g|)}$ is the product of the nonzero entries of generalized permutation matrix $g$.
Now, following Theorem \ref{wreath-reps}, 
each irreducible character of $\ZZ_r \wr S_n$ is of the form 
\[ \chi_\theta \overset{\mathrm{def}}= \Ind_{S_\theta}^{\ZZ_r \wr S_n} \( \bigodot_{i =0}^{r-1} \psi_i \wr {\theta_i} \),\qquad\text{where }S_\theta = \prod_{i=0}^{r-1} \ZZ_r \wr S_{|\theta_i|},\] for a unique $\theta \in \sP_r(n)$.  We refer to the $r$-partite partition $\theta$ of $n$ as
  the \emph{shape} of the irreducible character $\chi_\theta$.  The shape of an irreducible $\ZZ_r \wr S_n$-representation is then the shape of its character.

We recall also the following additional definitions from Section \ref{wreath-chars}:
\[ \ba \omega_k &= 
(1\cyc 2)(3\cyc 4)\cdots(2k-1\cyc 2k) \in S_{2k}, \\
 V_k^\tau &= \bigl\{ g \in \ZZ_r \wr S_{2k} : |g| \in C_{S_n}(\omega_k),\ z_{2i-1}(g)\cdot z_{2i}(g)= 1 \text{ for all }i\bigr\}.\ea\]
   The next theorem says precisely how to construct $\rho_{r,n}$ by inducing linear representations.  Its proof is simply an exercise in translating the notations of Theorem \ref{main-thm} and Lemma \ref{abelian-model-lem}.

\begin{theorem}\label{summary}
The wreath product $G_n= \ZZ_r \wr S_n$ has a generalized involution model with respect to the automorphism $g \mapsto (g^{-1})^T$.

\begin{enumerate}

\item[(1)] If $r$ is odd, then the model is given by the $1+\lfloor n/2\rfloor$ linear characters $ \lambda_k : C_{G_n,\tau}(\varepsilon_k) \to \QQ$ with $0\leq 2k \leq n $, where
\begin{enumerate}
\item[] $ \varepsilon_k= \(\barr{ll} \omega_k & 0 \\ 0 & I_{n-2k} \earr\)$, $0\leq 2k \leq n$,  
 are orbit representatives in $\cI_{G_n,\tau}$,
\item[]
\item[] $C_{G_n,\tau}(\varepsilon_k) = \left\{ g= \(\barr{cc} \nu & 0 \\ 0 & \pi \earr\) : \nu \in V_k^\tau,\ \pi \in S_{r-2k}\right\}$,
\item[]
\item[] $ \lambda_{k}(g) = \det(\nu)$ 
 for $g \in C_{G_n,\tau}(\varepsilon_k)$.

\end{enumerate}
If $\theta \in \sP_r(n)$ then the irreducible character $\chi_\theta$  is a constituent of  $\Ind_{C_{G_n,\tau}(\varepsilon_k)}^{G_n}(\lambda_k)$ if and only if the partitions $\theta_0$, $\theta_1$, $\dots$, $\theta_{r-1}$ have  $n-2k$ odd columns in total. 

\item[(2)] If $r$ is even,  then the model  is given by the $\lceil \frac{n+1}{2} \rceil \cdot \lfloor \frac{n+3}{2} \rfloor$ linear characters $ \lambda_{k,\ell} : C_{G_n,\tau}(\varepsilon_{k,\ell}) \to \QQ$ with $0\leq 2k+\ell \leq n$, where
\begin{enumerate}
\item[] $ \varepsilon_{k,\ell} =  \(\barr{lll} \omega_k & 0 & 0 \\ 0 & I_{n-2k-\ell} & 0 \\ 
0 & 0 & \zeta_r I_\ell \earr\)$, $0\leq 2k + \ell \leq n$,  
 are orbit representatives in $\cI_{G_n,\tau}$,
\item[]
\item[] $C_{G_n,\tau}(\varepsilon_{k,\ell}) = \left\{g= \(\barr{ccc} \nu & 0 & 0 \\ 0 & x & 0 \\ 0 & 0 & y \earr\) :
\nu \in V_k^\tau,\ x \in \ZZ_2 \wr S_{n-2k-\ell},\ y \in \ZZ_2 \wr S_\ell
\right\}$,
\item[]
\item[] $ \lambda_{k,\ell}(g) =\det(\nu) \det(y)/ \det(|y|)$ 
for $g \in C_{G_n,\tau}(\varepsilon_{k,\ell}).$
\end{enumerate}
If $\theta \in \sP_r(n)$ then the irreducible character $\chi_\theta$ is a constituent of  $\Ind_{C_{G_n,\tau}(\varepsilon_k)}^{G_n}(\lambda_k)$ if and only if the partitions $\theta_0$, $\theta_2$, $\dots$, $\theta_{r-2}$ have  $n-2k-\ell$ odd columns  in total and the partitions $\theta_1$, $\theta_3$, $\dots$, $\theta_{r-1}$ have $\ell$ odd columns in total.

%

\end{enumerate}

\end{theorem}

\begin{proof}
Assume $r$ is even; the case when $r$ is odd is the same but less complicated.  Let $\cI_r =  \ZZ_2 = \{ \pm 1\}$ denote the subgroup of involutions in $\ZZ_r$, and define $\One,\chi : \cI_r \to \CC$ to be the trivial and nontrivial characters of $\cI_r$, respectively.  By Lemma \ref{abelian-model},
\[ \Ind_{\cI_r}^{\ZZ_r} (\One) = \psi_0+\psi_2+\dots +\psi_{r-2}
\qquad\text{and}\qquad
 \Ind_{\cI_r}^{\ZZ_r} (\chi) = \psi_1+\psi_3+\dots +\psi_{r-1}.\]  
As in Section \ref{wreath}, let $\sU_2(n)$ denote the set of triples of nonnegative integers $x=(x_0,x_1,x_2)$ with $2x_0 + x_1 + x_2 = n$.  For each $x \in \sU_2(n)$ define $\phi_x^\tau : G_x^\tau \to \CC$ by (\ref{grp-rep-defs}) and $\varepsilon_x \in G$ as in Theorem \ref{main-thm},  where we take $H_1=H_2 = \cI_r$,  define $\tau$ by $\leftexp{\tau}g= (g^{-1})^T$, set $\varepsilon_1= 0 \in \ZZ_r$ and $\varepsilon_2 = 1 \in \ZZ_r$.   By Theorems \ref{thm1} and \ref{main-thm}, the linear characters  $\{ \phi_x^\tau : x \in \sU_2(n)\}$ form a generalized involution model for $G_n$, and $\chi_\theta$ is a constituent of  $\Ind_{G_x^\tau}^{G_n}(\phi_x^\tau)$ if and only if the partitions $\theta_0$, $\theta_2$, $\dots$, $\theta_{r-2}$ have  $x_1$ odd columns  in total and the partitions $\theta_1$, $\theta_3$, $\dots$, $\theta_{r-1}$ have $x_2$ odd columns in total.  The theorem is immediate after noting that $\varepsilon_x = \varepsilon_{x_0,x_2}$ and $\phi_x^\tau = \lambda_{x_0,x_2}$ in the notation of the current theorem, which follows easily from the fact that the product of the nonzero entries of an invertible generalized permutation matrix $g$ is precisely $\det(g) / \det(|g|)$.  
\end{proof}

In the following corollary, let $2\ZZ_r = \left\langle \zeta_r^2\right\rangle$, where $\zeta_r = e^{2\pi i / r}$ generates 
$\ZZ_r$.  If $r$ is odd then of course $2\ZZ_r = \ZZ_r$, while if $r$ is even then $2\ZZ_r=\ZZ_{r/2} = \{ 1=\zeta_r^0, \zeta_r^2,\dots,\zeta_r^{r-2}\}$.

\begin{corollary}\label{summary-cor}
Fix $\omega \in \ZZ_r \wr S_n$ such that $\omega = \omega^T$.  Let 
\[ \ba k& = \text{the number of 2-cycles in $|\omega|$}, \\
\ell &= \text{the number of $i \in \Fix(|\omega|)$ with $z_\omega(i) \notin 2 \ZZ_r$.}\ea\]
The character of 
the subrepresentation of $\rho_{r,n}$ generated by vector $C_\omega \in \cV_{r,n}$   is then the sum
$\sum_\theta \chi_\theta$  over all $\theta \in \sP_r(n)$ such that 

\begin{enumerate}
\item[(i)] When $r$ is odd, the partitions $\theta_0$, $\theta_1$, $\dots$, $\theta_{r-1}$ have $n-2k$ odd columns in total.

\item[(ii)] When $r$ is even, the partitions $\theta_0$, $\theta_2$, $\dots$, $\theta_{r-2}$ have $n-2k-\ell$ odd columns in total  and the partitions $\theta_1$, $\theta_3$, $\dots$, $\theta_{r-1}$ have $\ell$ odd columns in total.

\end{enumerate}
\end{corollary}

\begin{proof}
This follows from the preceding theorem after checking that the orbit of $\omega$ under the twisted conjugacy action $g : \omega \mapsto g \omega g^T$ contains $\varepsilon_k$ when $r$ is odd and $\varepsilon_{k,\ell}$ when $r$ is even.  
\end{proof}

\def\bflambda{\mbox{\boldmath$\lambda$}}
\def\bfP{\mbox{\boldmath$P$}}
\def\bfQ{\mbox{\boldmath$Q$}}

This corollary allows us to prove Conjecture 7.1 in \cite{APR2008}.  
Recall the definition given above of an $r$-partite partition of $n$.  One obtains an \emph{$r$-partite standard Young tableau} of shape $\theta \in \sP_r(n)$ by inserting the integers $1,2,\dots,n$ bijectively into the cells of the Ferrers diagrams of the partitions $\theta_0,\theta_1,\dots,\theta_{r-1}$ so that entries increase along each row and column of each partition.
%
  
  The natural subrepresentations considered in the preceding corollary have the following connection with  the generalized Robinson-Schensted correspondence for wreath products due to Stanton and White \cite{Stanton}.
  Recall, for example from \cite{Stanley}, that the usual  
   Robinson-Schensted-Knuth (RSK) correspondence is a bijective map \[\( \barr{cccc} a_1 & a_2 & \cdots & a_n \\ b_1 & b_2 & \cdots & b_n \earr\) \xrightarrow{\mathrm{RSK}} (P,Q)\] from two-line arrays of lexicographically ordered positive integers to pairs of semistandard Young tableaux $(P,Q)$ with the same shape.  Vewing $\sigma \in S_n$ as the two-line array with $a_i = i$ and $b_i = \sigma(i)$, 
this map restricts to a bijection from permutations to pairs of standard Young tableaux with the same shape.  
     Sch\"utzenberger proves in  \cite{Sch}  that the RSK correspondence associates to each involution $\omega \in \cI_{S_n}$ with $f$ fixed points a pair of standard Young tableaux $(P,Q)$ with $P=Q$ whose common shape has $f$ odd columns. 

  To define Stanton and White's  colored RSK correspondence for wreath products, fix an element $g \in \ZZ_r \wr S_n$ and for each $j \in [0,r-1]$, let $(P_j,Q_j)$ be the pair of tableaux obtained by RSK correspondence applied to the array 
   \be\label{array} \(\barr{cccc} i_1 & i_2 & \cdots & i_\ell \\ 
   \sigma(i_1) & \sigma(i_2) & \cdots & \sigma(i_\ell) \earr\)\ee where $\{ i_1 <i_2<\dots< i_\ell\}$   is the set of $i \in [1,n]$ with $z_g(i) = \zeta_r^j$.  
  The colored RSK correspondence is then the bijection from elements of $\ZZ_r \wr S_n$ to pairs of $r$-partite standard Young tableaux of the same shape defined by
  \[ g \longrightarrow (\bfP,\bfQ)= \biggl( (P_0, P_1,\dots, P_{r-1}), (Q_0, Q_1,\dots,Q_{r-1})\biggr).\]
To begin, we have the following easy corollary of Sch\"utzenberg's result.

\begin{lemma}\label{schu}
Fix $\omega \in \ZZ_r \wr S_n$ such that $\omega = \omega^T$ and suppose $\omega \mapsto (\bfP,\bfQ)$ under the colored RSK correspondence.  Then $\bfP = \bfQ$ and for each $j \in [0,r-1]$, the number of odd columns in the shape of $P_j$ is equal to the cardinality of $\left \{ i \in\Fix(|\omega|) : z_\omega(i) = \zeta_r^j \right\} $.
\end{lemma}

\begin{proof}
Since $\omega$ is a symmetric element, we have $z_\omega(i) = z_\omega(j)$ whenever $i$ and $j$ are in the same cycle of the involution $|\omega| \in S_n$.  Therefore each array (\ref{array}) corresponds to an involution in the group of permutations of the set $\{ i_1,\dots,i_\ell\}$, and it follows by Sch\"utzenberger's result that $\bfP=\bfQ$ and the number of odd columns in the shape of $P_j$ is as claimed.
\end{proof}
 
 We can now prove the theorem promised in the introduction.

\begin{theorem}  \label{conjecture}
Let $\cX$ be a set of symmetric elements in $\ZZ_r \wr S_n$.  If the elements of $\cX$ span a $\rho_{r,n}$-invariant subspace of $\cV_{r,n}$, then the subrepresentation of $\rho_{r,n}$ on this space is equivalent to the multiplicity-free sum of all irreducible $\ZZ_r\wr S_n$-representations whose shapes are obtained from the elements of $\cX$  by the colored RSK correspondence.


\end{theorem}

\begin{remark}
Caselli and Fulci prove a similar result concerning the decomposition of a different Gelfand model for $\ZZ_r \wr S_n$ in the recent preprint \cite{CF}.  Comparing the preceding theorem with \cite[Theorem 1.2]{CF} shows that there exist abstract isomorphisms between various natural subrepresentations of these two Gelfand models.
\end{remark}

The symmetric elements $\omega \in \ZZ_r \wr S_n$ whose underlying permutations $|\omega| \in S_n$ have a fixed number of 2-cycles form a union of twisted conjugacy classes with respect to the inverse transpose automorphsim, and so they span an invariant subspace of $\cV_{r,n}$. 
Hence, this result implies  \cite[Conjecture 7.1]{APR2008}.

\begin{proof}
It suffices to prove the theorem when $\cX = \left\{ g \omega g^T : g \in \ZZ_r \wr S_n\right\}$ is the orbit of some $\omega \in \ZZ_r \wr S_n$ with $\omega^T = \omega$.  
In this case, it follows by comparing Corollary \ref{summary-cor} and Lemma \ref{schu} that 
the colored RSK correspondence defines an injective map from $\cX$ to the set of $r$-partite standard Young tableaux whose shapes index irreducible constituents of the subrepresentation generated by $\cX$.  Since the number of such tableaux is equal to the cardinality of $\cX$ due to the well-known fact that the number of $r$-partite standard Young tableaux of shape $\theta$ is equal to $\chi_\theta(1)$, this map is in fact a bijection, which proves the theorem.
\end{proof}

We conclude by deriving two additional results which will be useful in the subsequent work \cite{?}.  Assume $r$ is even.  We then have two $\rho_{r,n}$-invariant subspaces of $\cV_{r,n}$ given by
\[\ba  \cV_{r,n}^+ &= \QQ\spanning\left\{ C_\omega : \omega \in \ZZ_r \wr S_n,\ \omega^T = \omega,\ \det(\omega) / \det(|\omega|) \in 2\ZZ_r \right\}, \\
 \cV_{r,n}^- &= \QQ\spanning\left\{ C_\omega : \omega \in \ZZ_r \wr S_n,\ \omega^T = \omega,\ \det(\omega) / \det(|\omega|) \notin 2\ZZ_r \right\}. \ea\] 
 Let $\chi_{r,n}^+$ and $\chi_{r,n}^-$ denote the characters of $\ZZ_r \wr S_n$ corresponding to the subrepresentations of $\rho_{r,n}$ on $\cV_{r,n}^+$ and $\cV_{r,n}^-$ respectively.

\begin{corollary}\label{plusminus}
Let $r,n$ be positive integers with $r$ even.  Given $\theta \in \sP_r(n)$, define $\Omega(\theta) $ as  the sum of the numbers of odd columns in the partitions $\theta_1,\theta_3,\dots, \theta_{r-1}$.  Then 
\[ \chi_{r,n}^+ = \sum_{\substack{\theta \in \sP_r(n), \\ \Omega(\theta)\text{ is even}}}\chi_\theta
\qquad\text{and}\qquad
\chi_{r,n}^- = \sum_{\substack{\theta \in \sP_r(n), \\ \Omega(\theta)\text{ is odd}}}\chi_\theta.
\]

\end{corollary}

\begin{proof}
Since $\det(\omega) / \det(|\omega|) \in 2\ZZ_r$ for a symmetric element $\omega \in \ZZ_r \wr S_n$ if and only if the union of the disjoint sets $ \left\{ i \in \Fix(|\omega|) : z_\omega(i) = \zeta_r^j\right\}$ over all odd $j \in [0,r-1]$ has even cardinality, this is immediate from Lemma \ref{schu} and Theorem \ref{conjecture}.
%
\end{proof}

Suppose $p$ is a   positive integer dividing $r$.  Let $\gamma : \ZZ_r \wr S_n \to \CC$ denote the linear character defined by 
\[\gamma(g) = \(\psi_{r/p}\wr (n)\)(g) =   \( \frac{\det(g) }{ \det(|g|)}\)^{r/p},\qquad\text{for }g\in \ZZ_r \wr S_n.\] Here $(n)$ denotes the trivial partition of $n$.  
A straightforward calculation shows that for all $\theta \in \sP_r(n)$ we have 
\be\label{otimes} \gamma \otimes \chi_\theta = \chi_{\theta'},\qquad\text{where}\qquad \theta'_i = \theta_{i-r/p}\text{ for }i\in [0,r-1]\ee where with slight abuse of notation we define $\theta_{i-r} = \theta_i$ for $i \in [0,r-1]$. 
This observation leads to the following corollary of Lemma \ref{plusminus}.

\begin{proposition} \label{plusminus-cor} Let $r,p,n$ be positive integers with $r$ even and $p$ dividing $r$. Then
\[ \gamma\otimes \chi_{r,n}^+= \left\{\ba & \chi_{r,n}^-, &&\text{if $n$ and $r/p$ are odd}, \\
&\chi_{r,n}^+,&&\text{otherwise},\ea\right.
\qquad
\gamma\otimes \chi_{r,n}^- =\left\{\ba & \chi_{r,n}^+, &&\text{if $n$ and $r/p$ are odd}, \\
&\chi_{r,n}^-,&&\text{otherwise}.\ea\right.
\]

\end{proposition}

\begin{proof}
Recall the definition of $\Omega$ from Corollary \ref{plusminus} and let $\Omega'(\theta)$ for $\theta \in \sP_r(n)$ be the sum of the numbers of odd columns in the partitions $\theta_0,\theta_2,\dots,\theta_{r-2}$.  Suppose $r/p$ is odd; then (\ref{otimes}) implies that the map $\chi\mapsto \gamma\otimes \chi$ exchanges the two sets \be\label{set}
\{ \chi_\theta : \theta \in \sP_r(n),\ \Omega(\theta)\text{ is odd} \}\qquad\text{and}\qquad \{ \chi_\theta : \theta \in \sP_r(n),\ \Omega'(\theta)\text{ is odd} \}.\ee  If $n$ is odd, then $\theta \in \sP_r(n)$ has $\Omega'(\theta)$ odd if and only if $\Omega(\theta)$ is even, and it follows immediately from Corollary \ref{plusminus} that $\gamma\otimes \chi_{r,n}^\pm =  \chi_{r,n}^\mp$.  
If $n$ is even, then $\theta \in \sP_r(n)$ has $\Omega'(\theta)$ odd if and only if $\Omega(\theta)$ is odd, so the two sets in (\ref{set}) are the same, and necessarily $\gamma \otimes \chi_{r,n}^+ = \chi_{r,n}^+$.  
Alternatively, if $r/p$ is even, then by (\ref{otimes}) the map $\chi\mapsto \gamma\otimes \chi$ defines a permutation of the set of irreducible constituents of $\chi_{r,n}^+$ so $\gamma \otimes \chi_{r,n}^+ = \chi_{r,n}^+$.  Similar arguments show that $\gamma \otimes \chi_{r,n}^- = \chi_{r,n}^-$ if $n$ or $r/p$ is even.
\end{proof}

We continue this discussion and apply these results in the complementary work \cite{?}, where we show how and when the Gelfand model $\rho_{r,n}$ can be extended to the complex reflection group $G(r,p,n)$, and classify the finite complex reflection groups which have generalized involution models.

\end{document}